\documentclass[11pt,reqno]{amsart}
\usepackage{fullpage}
\usepackage[utf8]{inputenc}
\usepackage[colorlinks,
            linkcolor=red,
            anchorcolor=red,
            citecolor=blue
            ]{hyperref}

\usepackage{graphicx, amsmath, fullpage, amssymb, amsthm, amsfonts, color, algorithm,mathtools,bbm}
\mathtoolsset{showonlyrefs}
\usepackage{enumerate}
\usepackage{cases}
\newtheorem{theorem}{Theorem}[section]
\newtheorem{lemma}[theorem]{Lemma}
\usepackage{comment}
\usepackage{appendix}
\usepackage{tikz}
\usepackage{tkz-berge}
\usepackage{graphicx}
\usepackage{array}
\usepackage{makecell}
\usepackage{mathtools}

\theoremstyle{definition}
\newtheorem{definition}[theorem]{Definition}

\newtheorem{assumption}[theorem]{Assumption}

\newtheorem{prop}[theorem]{Proposition}

\newtheorem{cor}[theorem]{Corollary}
\newtheorem{remark}[theorem]{Remark}

\newcommand{\Cov}{\mathrm{Cov}}

\title{Central limit theorems for linear spectral statistics of  inhomogeneous random graphs with graphon limits}
\author{Xiangyi Zhu}
\address{Department of Applied Mathematics and Statistics, Johns Hopkins University,
Baltimore, MD 21218}
\email{xzhu96@jh.edu}

\author{Yizhe Zhu}
\address{Department of Mathematics, University of Southern California, Los Angeles, CA 90007}
\email{yizhezhu@usc.edu}

\begin{document}

\maketitle

\begin{abstract}
We establish central limit theorems (CLTs) for the linear spectral statistics of the adjacency matrix of inhomogeneous random graphs across all sparsity regimes, providing explicit covariance formulas under the assumption that the variance profile of the random graphs converges to a graphon limit. Two types of CLTs are derived for the (non-centered) adjacency matrix and the centered adjacency matrix, with different scaling factors when the sparsity parameter $p$ satisfies $np = n^{\Omega(1)}$, and with the same scaling factor when $np = n^{o(1)}$. In both cases,  the limiting covariance is expressed in terms of homomorphism densities from certain types of finite graphs to a graphon. These results highlight a phase transition in the centering effect for global eigenvalue fluctuations. For the non-centered adjacency matrix, we also identify new phase transitions for the CLTs in the sparse regime when $n^{1/m} \ll np \ll n^{1/(m-1)}$ for $m \geq 2$. Furthermore, weaker conditions for the graphon convergence of the variance profile are sufficient as $p$ decreases from being constant to $np \to c\in (0,\infty)$. These findings reveal a novel connection between graphon limits and linear spectral statistics in random matrix theory.
\end{abstract}

\section{Introduction}

\subsection{Inhomogeneous random matrices and graph limits}
The spectrum of inhomogeneous random matrices is an important topic in random matrix theory. In recent years, much work has been done in terms of its global spectral distribution \cite{nica2002operator,zhu2020graphon}, local law and universality \cite{ajanki2019quadratic,ajanki2017universality, dumitriu2019sparse},  extreme eigenvalues  \cite{benaych2020spectral,dumitriu2024extreme,benaych2019largest,hiesmayr2023spectral}, and mesoscopic spectral CLTs \cite{landon2024single,kruger2023mesoscopic}. A special class of inhomogeneous random matrices is the adjacency matrix of sparse inhomogeneous random graphs, which is of particular interest in spectral random graph theory, with applications to network estimation \cite{abbe2023learning,stephan2022non,gao2021minimax,banerjee2017optimal}, community detection \cite{abbe2018community,le2018concentration,stephan2024sparse}, matrix completion \cite{bordenave2023detection,stephan2024non}, and graph neural networks \cite{ruiz2021graphon,ruiz2023transferability,keriven2020convergence}.

There is a recent line of work aiming to connect graph limit theory with random matrix theory. Various types of graph limits have proven instrumental in deriving spectral properties of random matrices. For example, local weak convergence has been a key tool for studying the spectral properties of random graphs with bounded expected degrees and heavy-tailed random matrices, as shown in \cite{bordenave2010resolvent,enriquez2016spectra,bordenave2011rank,jung2018delocalization,bordenave2024large}. Similarly, action convergence \cite{backhausz2022action} and local-global graph limits \cite{backhausz2019almost} have provided new perspectives on understanding graph-based random matrices.

The interplay between graphon theory and random matrix theory has also been useful for the study of the limiting spectral distributions of various matrix models. For instance, Wigner-type matrices have been analyzed through graphon theory in \cite{zhu2020graphon,sanders2023singular}, while random Laplacian matrices have been investigated in \cite{chakrabarty2021spectra,chatterjee2022spectral}. 
These approaches have further developed to find the limiting spectra of inhomogeneous random graphs with bounded expected degrees \cite{avena2023limiting} and provided convergence results for spectral norms \cite{cheliotis2024limit}. The graphon framework has also been used recently to study the large deviation of empirical spectral measures of sparse random graphs \cite{augeri2024large} and in the study of nonlinear random matrices \cite{dabo2024traffic} and non-commutative probability theory \cite{cebron2024graphon}.

Together, these developments highlight the deep connections between graph limit frameworks and random matrix theory, offering powerful tools to tackle both theoretical and applied problems in the study of sparse and inhomogeneous random graphs.
In this paper, we establish a novel connection between graphon theory and the linear spectral statistics of inhomogeneous random graphs, extending beyond the graphon-based approach to the empirical spectral distributions of inhomogeneous random graphs presented in \cite{zhu2020graphon}.

\subsection{Global eigenvalue fluctuations}
The study of fluctuations in the limiting empirical spectral distributions (ESDs) for random matrices is a well-established topic in random matrix theory, with origins traced back to foundational works such as \cite{jonsson1982some, khorunzhy1996asymptotic, sinai1998central}; see also \cite{anderson2010introduction} and references therein. In recent years, this research has been extended to sparse random matrices and random graph-related matrices under various sparsity regimes and dependence structures \cite{shcherbina2010central, shcherbina2012central, benaych2014central, dumitriu2013functional, ben2015fluctuations, dumitriu2023global, kanazawa2023central}. 

Linear spectral statistics refer to the quantities of the form
$\mathcal{L}(f)=\sum_{i=1}^n f(\lambda_{i})$,
where $\lambda_{i}, 1\leq i\leq n$ are eigenvalues of a Hermitian matrix $M$ and $f$ is a suitable test function. In this paper, we focus on monomial test functions $f(x)=x^k$, where $M$ is a random matrix associated with a random graph.
 With additional effort, the class of test functions can be extended to a broader range of functions \cite{dumitriu2012global,shcherbina2010central,shcherbina2012central}.
 
For Wigner matrices with all finite moments, no additional scaling factor is required to obtain a central limit theorem (CLT) (see, e.g.,\cite{anderson2010introduction}) for $\mathcal L(f)$. Namely, when centered, for suitable smooth function $f$, $\mathcal L(f)$ converges in distribution to a normal distribution whose variance depends on $f$
\begin{align}
   \mathcal{L}(f)-\mathbb E \mathcal{L}(f)\to N(0,\sigma_{f}^2). 
\end{align}
However, for random matrices that are sparse or have heavy-tailed entries, additional scaling factors become necessary \cite{benaych2014central,shcherbina2010central,shcherbina2012central}. This is also the case for sparse inhomogeneous random graphs considered in this paper; see Section~\ref{sec:main} for details.

We also review existing results for the CLT of the linear spectral statistics of  homogeneous Erd\H{o}s-R\'{e}nyi graphs random graphs $G(n,p)$. For the centered adjacency matrix, CLTs for $\mathcal L(f)$ after proper normalization have been established for $p\to 0$ and  $np\to\infty$ in \cite{shcherbina2012central} and for $p=c/n$ in \cite{shcherbina2010central,benaych2014central}; see the survey \cite{guionnet2021bernoulli} for more details.  For global eigenvalue fluctuations of the non-centered adjacency matrix $A_n$, the fluctuation of the largest eigenvalue dominates those of the bulk eigenvalues. Consequently, the top eigenvalue makes the primary contribution, as shown in \cite{diaconu2022two}, illustrating the centering effect when considering global eigenvalue fluctuations of $A_n$.

\subsection{Contribution of this paper}
We present a unified approach to studying the global eigenvalue fluctuations of the adjacency matrix of inhomogeneous random graphs across all sparsity regimes. We summarize our main theorems in Table~\ref{fig:summary}. Notably, the derived limits vary significantly depending on the sparsity level and whether the adjacency matrix is centered or not. Below, we provide a detailed discussion of these contributions.

\begin{small}
\begin{table}
 \renewcommand{\arraystretch}{1.2}
\begin{tabular}{|c|c|c|c|c|}
\hline
Sparsity & CLT for $\overline{A}_n$ & CLT for $A_n$ & Scaling & Graphon convergence \\ 
\hline
$p\in (0,1)$ is a constant  & \makecell{Theorem~\ref{thm: centered_CLT_inhomo_dense}} & \makecell{Theorem~\ref{thm:nc_CLT_inhomo_dense}} & \makecell{Different} & \makecell{$\delta_1(W_n, W)\to 0$} \\ 
\hline
$p\to 0, np=n^{\Omega(1)}$ & \makecell{Theorem~\ref{thm: centered_CLT_inhomo_sparse1}} & \makecell{Parts (1)-(4) of \\ Theorem~\ref{thm:nc_CLT_inhomo_sparse}} & \makecell{Different} & \makecell{For $\overline{A}_n$: $t(T,W_n)$ converges; \\ For ${A}_n$, $\delta_{\Box}(W_n,W)\to 0$}\\ 
\hline 
 $np\to\infty, np=n^{o(1)}$ & \makecell{Theorem~\ref{thm: centered_CLT_inhomo_sparse1}} & \makecell{Part (5) of \\ Theorem~\ref{thm:nc_CLT_inhomo_sparse}} & \makecell{Same} & \makecell{ $t(T,W_n)$ converges} \\
\hline
$np\to c\in (0,\infty)$ & \makecell{Theorem~\ref{thm: inhomo_bounded_c}} & \makecell{Theorem~\ref{thm:nc_CLT_inhomo_bounded}}  & \makecell{Same} & \makecell{$t(T,W_n)$ converges}\\ 
\hline
\end{tabular}
\vspace{0.1 in}
\caption{Summary of main results. For different sparsity regimes of $p$, we compare the CLTs for $A_n$ and $\overline{A}_n$ in different theorems, along with the required graphon convergence assumptions. When $np = n^{\Omega(1)}$, the scaling factors for $A_n$ and $\overline{A}_n$ differ (see Equations~\eqref{eq:Xnk} and \eqref{eq:Lnk}), whereas for $np = n^{o(1)}$, the scaling factors are the same (see Equations~\eqref{eq:Xnk} and \eqref{eq:newlimit}). Additionally, weaker graphon convergence requirements are sufficient as $p$ decreases.}\label{fig:summary}
\end{table}
\end{small}

\subsubsection{Linear spectral statistics of Wigner-type matrices}
Wigner-type random matrices \cite{ajanki2017universality} refer to the Hermitian random matrix model where each entry is independent (up to symmetry) with different variances. Significant interest has focused on their eigenvalue fluctuations.  \cite[Theorem 4.2]{chatterjee2009fluctuations} established a CLT without an explicit variance under additional distribution assumptions, later generalized in \cite{adhikari2021linear}. 
 \cite{wang2023central} studied CLT with an explicit limiting variance for dense block Wigner-type matrix whose zero diagonals,  a special case of our model.   \cite{anderson2006clt} derived a CLT for polynomial test functions for matrices with variance profiles from discretized continuous functions on $[0,1]^2$. In the dense case, \cite{landon2024single} established a global CLT for smooth functions under additional assumptions, with explicit variance.

In this paper, we establish the first CLT for the linear spectral statistics of sparse Wigner-type matrices, covering all sparsity regimes. Our results compute the explicit covariance of the limiting Gaussian process using graphon theory.

\subsubsection{The centering effect}
For  $G(n,p)$, the bulk eigenvalues of the adjacency matrix $A_n$ and the centered adjacency matrix $\overline{A}_n$ are similarly \cite{tran2013sparse,erdHos2013spectral,alt2021extremal}, since $\mathbb E A_n$ is a rank-1 perturbation.
For inhomogeneous random graphs, when $np=\Omega(\log^c n)$ and $\mathbb EA_n$ is finite-rank, it was shown in \cite{chakrabarty2020eigenvalues,lunde2023subsampling} that outlier eigenvalues are asymptotically jointly Gaussian, generalizing the homogeneous case in \cite{erdHos2012spectral}.  

When $\mathbb E A_n$ is of full rank, one can not treat it as a finite-rank perturbation of a Wigner-type matrix. Without the low-rank assumption on the graphon $W$, the convergence of $\lambda_i(A_n)$ for every fixed $i$ of a dense random graph with a graphon limit was shown in \cite{borgs2012convergent}. Fluctuation of the largest eigenvalue of the adjacency matrix of a dense $W$-random graph was recently studied in \cite{chatterjee2024fluctuation}, whose limiting distribution, depending on the property of the graphon $W$, can be non-Gaussian.

 We demonstrate that without any low-rank assumption on the graphon $W$, after proper normalization, a CLT for $A_n$  holds with an explicit variance in all sparsity regimes.  Additionally, we identified a phase transition of the centering effect:
 \begin{itemize}
\item For $np=n^{\Omega(1)}$, different CLTs hold for $A_n$ and $\overline{A}_n$ with different scaling factors.
\item For $np=n^{o(1)}$, both matrices follow the same CLT with identical scaling factors. 
 \end{itemize}
 See Section~\ref{sec:main} for more details. This presents the first results in the literature on the CLT for the linear spectral statistics of $A_n$ without centering in inhomogeneous random graphs across all sparsity regimes.

In the sparse regime ($np\to\infty, p\to 0$),  we identify infinitely many phase transitions for as $p$ decreasing. For $ n^{\frac{1}{m}}\ll np\ll n^{\frac{1}{m-1}}$ ($m\geq 2$) and   $n^{1-\frac{1}{m}}p\to c$, different limiting covariances arise. See Theorem~\ref{thm:nc_CLT_inhomo_sparse} for more details.  Moreover, in the bounded-degree regime,  our covariance formula in Theorems~\ref{thm: inhomo_bounded_c} and~\ref{thm:nc_CLT_inhomo_bounded} recover the one for the centered adjacency matrix of Erd\H{o}s-R\'enyi graphs in \cite[Theorem 1.4]{benaych2014central}.

\subsubsection{Three types of graphon convergence}
All results for the CLTs in this paper assume the variance profile matrix $S_n$ of the inhomogeneous random graph converges to limit under various types of graphon convergence. The limiting covariance of the CLT is expressed in terms of homomorphism densities from finite graphs to a graphon.

A notable feature of our results is that the required topology for graphon convergence becomes weaker as the sparsity parameter $p$ decreases. This occurs because the finite graphs contributing to the limiting variance term become simpler in sparser regimes.
\begin{itemize}
    \item Dense regime ($p\in (0,1)$): A stronger notion of convergence, called $\delta_1$-convergence (see Assumption~\ref{assumption:delta1}), is needed. This is due to the influence of multigraph homomorphism densities on the variance term.
    \item Sparse regime ($np\to \infty, p\to 0$): Multigraph homomorphism densities no longer contribute to the limiting variance, and convergence of tree homomorphism suffices for the CLT of $\overline{A}_n$. While for $A_n$, $\delta_{\Box}$ convergence in required when $np=n^{\Omega(1)}$ (Assumption~\ref{assumption:deltaBox}).  
  \item For the CLT of $\overline{A}_n$ when $np\to c$, and the CLT of $A_n$ when $np=n^{o(1)}$, the only nonzero covariance in the limit arises from homomorphism densities of trees to a graphon. Consequently, only a weaker assumption on finite tree homomorphism densities is required (Assumption~\ref{assumption:tree}), as shown in \cite{zhu2020graphon}.
\end{itemize}

Furthermore, by adapting the proof in \cite[Theorem 5.1]{zhu2020graphon},  all results  in Section~\ref{sec:main} are also valid for $W$-random graphs where the inhomogeneous random graph is generated such that each edge  $\{i,j\}$ is connected independently with probability  $p W(x_i,x_j)$, with $x_1,\dots,x_n$ i.i.d. samples uniformly distributed on $[0,1]$.

\subsection*{Organization of the paper}
    The rest of the paper is organized as follows. Section~\ref{sec:prelim} introduces key definitions and lemmas from graphon theory. In Section~\ref{sec:main}, we present the main results for the centered and non-centered adjacency matrices of inhomogeneous random graphs across different sparsity regimes. 
Section~\ref{sec:proof_centered} contains the proofs of the CLTs for the centered adjacency matrix, while Section~\ref{sec:proof_adj} focuses on the proofs of the CLTs for the non-centered adjacency matrix.
Appendix~\ref{sec:appendix} provides the proofs of several technical lemmas, and Appendix~\ref{sec:size} includes additional results on the sizes of graph classes defined in Definitions~\ref{def: graph_tree} and~\ref{def: TC}.

\subsubsection*{Notation} For two sequences $f_n,g_n\geq 0$ we denote $f_n\ll g_n$ if $\frac{f_n}{g_n}\to 0$ as $n\to\infty$. 


\section{Graphon theory}\label{sec:prelim}
In this section, we introduce fundamental definitions from graphon theory. For more details, see \cite{lovasz2012large}. A \textit{graphon} is a symmetric, integrable function $W:[0,1]^2\to \mathbb R$.	
Here symmetric means $W(x,y)=W(y,x)$ for all $x,y\in[0,1]$.

Every weighted graph $G$ has an associated graphon $W^G$ constructed as follows: First, divide the interval $[0,1]$ into subintervals $I_1,\dots, I_{|V(G)|}$, each of length $\frac{1}{|V(G)|}$. Then give the edge weight  $\beta_{ij}$ on $I_i\times I_j$, for all $i,j\in V(G)$. In this way, every finite weighted graph gives rise to a graphon.

The most important metric on the space of graphons is the \textit{cut metric}. The space of all graphons taking values in $[0,1]$, endowed with the cut metric, forms a compact metric space.

\begin{definition}[cut norm]\label{cutmetric}
For a graphon $W:[0,1]^2\to\mathbb R$, the \textit{cut norm} is defined as
\begin{align*}
	\|W\|_{\Box}:=\sup_{S,T\subseteq [0,1]}\left|\int_{S\times T} W(x,y)dx dy\right|,
\end{align*}
	where $S$, $T$ range over all measurable subsets of $[0,1]$. For two graphons $W, W': [0,1]^2\to\mathbb R$, we define 
$
	d_{\Box}(W,W'):=\|W-W'\|_{\Box}
$.
The \textit{cut metric} $\delta_{\Box}$ is then given by
\begin{align}
\delta_{\Box}(W,W'):=\inf_{\sigma} d_{\Box}(W^{\sigma},W'),	 \notag
\end{align}
where $\sigma$ ranges over all measure-preserving bijections $[0,1]\to [0,1]$ and 
$W^{\sigma}(x,y):=W(\sigma(x),\sigma(y))$.
\end{definition}
The cut metric provides a notion of convergence for graph sequences, with the limiting object being the graphon.

Similarly,  the $\delta_1$-distance  is defined as  
\begin{align}
\delta_{1}(W,W'):=\inf_{\sigma} \|W^{\sigma}-W'\|_1=\inf_{\sigma}  \int_{[0,1]^2} |W^{\sigma}(x,y)-W'(x,y)|dxdy.	 \notag
\end{align}
By definition, for any two graphons $W,W'$, $\delta_{\Box}(W,W')\leq \delta_1(W,W')$.

Another way of defining the convergence of graphs is to consider graph homomorphisms.
  
 \begin{definition}[homomorphism density]\label{def: homomorphism density}
 	For any graphon $W$ and multigraph $F=(V,E)$ (without loops), define the \textit{homomorphism density} from $F$ to $W$ as 
\begin{align}
t(F,W):=\int_{[0,1]^{|V|}} \prod_{ij\in E}W(x_i,x_j) \prod_{i\in V}dx_i.	\notag
\end{align}
 \end{definition}


It is natural to think two graphons $W$ and $W'$ are similar if they have similar homomorphism densities from any finite graph $G$. This leads to the following definition of left convergence.

\begin{definition}[left convergence]
	Let $W_n$ be a sequence of graphons. We say  $W_n$ is \textit{convergent from the left} if $t(F,W_n)$ converges for any finite simple graph $F$ (without loops or multi-edges). 
\end{definition}
The importance of homomorphism densities lies in their ability to characterize convergence under the cut metric.
Let $\mathcal W_0$ be the set of all graphons such that $0\leq W\leq 1$.
The following is a characterization of convergence in the space $\mathcal W_0$.

\begin{lemma}[Theorem 11.5 in \cite{lovasz2012large}] \label{graphonconvergence}
Let $\{W_n\}$ be a sequence of graphons in $\mathcal W_0$ and let $W\in \mathcal W_0$.	Then  $t(F,W_n)\to t(F,W)$ for all finite simple graphs if and only if $\delta_{\Box}(W_n,W)\to 0$.
\end{lemma}
The following lemma shows that the homomorphism density function is continuous under the cut metric only when $F$ is a simple graph.
\begin{lemma}[Lemma C.2, \cite{janson2010graphons}]\label{lem: graphon_simple}
    The mapping $W \mapsto t(F,W)$ is continuous on $(\mathcal{W}_0, \delta_{\square})$ if and only if F is a simple graph.
\end{lemma}
To guarantee convergence of homomorphism density for loopless multigraphs, we need a stronger notation of convergence under the $\delta_1$ metric. Lemma~\ref{lem: graphon_multi} was stated without a proof in \cite{janson2010graphons}. We provide its proof in Appendix~\ref{sec:appendix} for completeness. 
\begin{lemma}[Lemma C.4, \cite{janson2010graphons}]\label{lem: graphon_multi}
    The mapping  $W \mapsto t(F,W)$ is continuous on $(\mathcal{W}_0, \delta_{1})$ for every loopless multigraph.
\end{lemma}
 The following lemma for the convergence of the graphon sequence $W_n'=W_n(1-pW_n)$ under $\delta_1$ metric is also needed in our proof. We provide its proof in Appendix~\ref{sec:appendix}.
\begin{lemma} \label{lem: W'}
Assume $p\in [0,1]$.
Let $W_n'=W_n(1-pW_n), W'=W(1-pW)$.
    Suppose $\delta_1(W_n, W)\to 0$, then $\delta_1(W_n',W')\to 0$.
\end{lemma}

\section{Main results}\label{sec:main}
In this paper, we study the spectrum of Hermitian random matrices under the following assumptions:
\begin{assumption}[Inhomogeneous random graphs]\label{assumption1}
Assume $A_n$ is  the adjacency matrix of a random graph satisfying the following conditions:
    \begin{enumerate}
    \item  The entries $a_{ij}$ are independent (up to symmetry) with $a_{ii}=0$, for all $i\in [n]$.
    \item $a_{ij} \sim \text{Ber}(p s_{ij})$, where $p\in (0,1)$ is the sparsity parameter and $S_n=(s_{ij})$is the variance profile matrix.
    \item The entries in $S_n$ satisfy $0\leq s_{ij}\leq C$ for an absolute constant $C>0$ and  $\sup_{ij} (ps_{ij}) \leq 1$. \label{assump: dense_1}
\end{enumerate}
\end{assumption}

The sparsity parameter $p$ determines the regime of the model:
\begin{itemize}
    \item \textit{Dense regime}:  $p$ is a constant independent of $n$.
    \item \textit{Sparse regime}: $p\to 0$ and $np\to\infty$.
    \item \textit{Bounded expected degree regime}:  $p\to 0$ and $np\to  c\in (0,\infty)$.
\end{itemize}

We associate a graphon $W_n$ to the matrix $S_n$. Treat $S_n$ as the adjacency matrix of a weighted complete graph such that the weight of each edge is $s_{ij}$. The graphon $W_n$ is defined as the corresponding graphon representation of $S_n$. 
The centered adjacency matrix $\overline{A}_n$ is defined  as 
$\overline{A}_n=A_n-\mathbb E A_n$.

When the variance profile matrix $S_n$  has a graphon limit, we show that for monomial test functions $f(x)=x^k$, the linear spectral statistics of  $A_n$ and $\overline{A}_n$, after proper rescaling, converge to centered Gaussian processes with an explicit covariance structure. 

  For different sparsity regimes and depending on whether $A_n$ is centered or not, different graphon convergence assumptions are required, as described below.

\begin{assumption}[$\delta_1$-convergence]\label{assumption:delta1}
    Assume $\delta_1(W_n, W) \rightarrow 0$ for some graphon $W$ as $n \rightarrow \infty$.
\end{assumption}

\begin{assumption}[$\delta_{\Box}$-convergence]\label{assumption:deltaBox}
    Assume $\delta_{\square}(W_n, W) \rightarrow 0$ for some graphon $W$ as $n \rightarrow \infty$.
\end{assumption}

\begin{assumption}[Convergence of tree homomorphism densities]\label{assumption:tree}
Assume for any finite tree $T$, there exists a constant $\beta_T\geq 0$ such that $t(T, W_n)\to \beta_T$  as $n \rightarrow \infty$.
\end{assumption}

\begin{remark}[Hierarchy of graphon convergence assumptions]
    Note that Assumption~\ref{assumption:delta1} is the strongest assumption, which ensures convergence of homomorphism densities of finite loopless multigraphs (Lemma~\ref{lem: graphon_multi}). Assumption~\ref{assumption:deltaBox} is weaker than $\delta_1$-convergence, but sufficient for convergence of homomorphism densities of finite simple graphs (Lemma~\ref{graphonconvergence}). Assumption~\ref{assumption:tree}  is the weakest, which only guarantees the convergence of finite tree homomorphism densities.
\end{remark}

We present the main theorems for $\overline{A}_n$ and $A_n$ in Sections~\ref{sec:centered} and \ref{sec:non_centered}, respectively.

\subsection{Centered adjacency matrices}\label{sec:centered}

 We consider  the scaled linear spectral statistics  defined as 
\begin{align}\label{eq:Xnk}
    X_{n,k} = \sqrt{p} \cdot \left( \mathrm{tr}\left(\frac{\overline{A}_n}{\sqrt{np}}\right)^k - \mathbb{E} \mathrm{tr}\left(\frac{\overline{A}_n}{\sqrt{np}}\right)^k \right).
\end{align}
For  $k=1$, $X_{n,1}=0$ since $A_n$ is diagonal-free. Thus, we assume $ k \geq 2$ throughout the paper. 

To describe the limiting variance, we introduce the following graph classes.

\begin{definition}[Two rooted trees with overlapping edges]\label{def: graph_tree}
    For any even integers $k,h\geq 2$, we define $\mathcal T^{k,h}_1 $  as a set of planar trees with $ |V| = |E|+1 =(k+h)/2$, constructed by joining two rooted planar trees of length $k/2$ and $h/2$ with one overlapping edge. See the left part of Figure~\ref{fig:dense_union_tree_center} for an example.    
    
    Similarly, We define \( \mathcal T^{k,h}_2 \) as a set of multigraphs $ \mathcal T^{k,h}_2 = (V, E)$ with \(|V| = |E| = (k+h)/2 \), formed by joining two rooted planar trees of length $ k/2$ and $ h/2$, sharing two vertices connected by two distinct edges. See the right part of Figure~\ref{fig:dense_union_tree_center} for an example.

   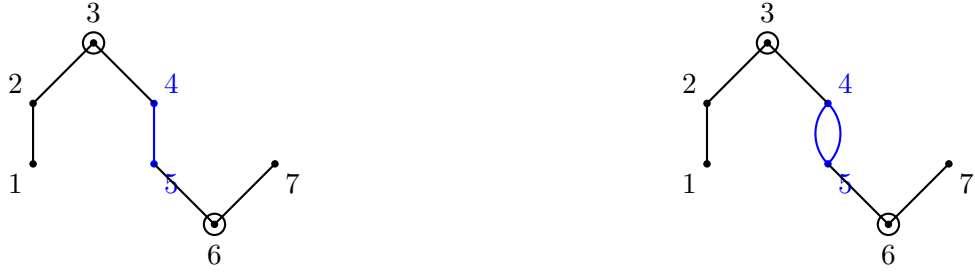
\begin{figure}
     \begin{minipage}{0.46\textwidth}
    \centering
    \begin{tikzpicture}[scale=0.8]
        \filldraw[black] (0, 1) circle (1.5pt) node[below left] {1};
        \filldraw[black] (0, 2) circle (1.5pt) node[above left] {2};
        \filldraw[black] (1, 3) circle (1.5pt) ;
        \draw[black, thick] (1,3) circle (5 pt) node[above=4pt] {3};
        \filldraw[blue] (2, 2) circle (1.5pt) node[above right] {4};
        \filldraw[blue] (2, 1) circle (1.5pt) node[below right] {5};
        \filldraw[black] (3, 0) circle (1.5pt) ;
        \draw[black, thick]  (3, 0) circle (5 pt) node[below=4pt] {6};
        \filldraw[black] (4, 1) circle (1.5pt) node[below right] {7};

        \draw[thick] (0, 1) -- (0, 2);
        \draw[thick] (0, 2) -- (1, 3) ;
        \draw[thick] (1, 3) -- (2, 2); 
        \draw[thick,blue] (2, 2) -- (2, 1); 
        \draw[thick] (2, 1) -- (3, 0); 
        \draw[thick] (3, 0) -- (4, 1); 
    \end{tikzpicture}
    \end{minipage}
    \hfill 
    \begin{minipage}{0.46\textwidth}
    \centering
    \begin{tikzpicture}[scale=0.8]
        \filldraw[black] (0, 1) circle (1.5pt) node[below left] {1};
        \filldraw[black] (0, 2) circle (1.5pt) node[above left] {2};
       \filldraw[black] (1, 3) circle (1.5pt) ;
        \draw[black, thick] (1,3) circle (5 pt) node[above=4pt] {3};
        \filldraw[blue] (2, 2) circle (1.5pt) node[above right] {4};
        \filldraw[blue] (2, 1) circle (1.5pt) node[below right] {5};
         \filldraw[black] (3, 0) circle (1.5pt) ;
        \draw[black, thick]  (3, 0) circle (5 pt) node[below=4pt] {6};
        \filldraw[black] (4, 1) circle (1.5pt) node[below right] {7};

        \draw[thick] (0, 1) -- (0, 2);
        \draw[thick] (0, 2) -- (1, 3) ;
        \draw[thick] (1, 3) -- (2, 2); 
        \draw[thick] (2, 1) -- (3, 0); 
        \draw[thick] (3, 0) -- (4, 1); 
         \draw[thick,blue] (2, 2)  to[out=225, in=135] (2, 1);
        \draw[thick,blue] (2, 2) to[out=315, in=45] (2, 1);  
    \end{tikzpicture}
    \end{minipage}
    \caption{We give an example of \( k=8 \) and $h = 6$. The left figure is an element in \(\mathcal T^{k,h}_1 \) constructed by a rooted tree one of length 4 with edges $\{1,2\}, \{2,3\}, \{3,4\}, \{4,5\}$ and the root vertex $3$,  and another rooted tree of length 3 with edges $ \{4,5\}, \{5,6\}, \{6,7\}  $ and a root vertex $6$, and one overlapping edge $\{4,5\} $.  The right figure is a corresponding element in \( \mathcal T^{k,h}_2 \)  constructed by two rooted planar trees, one of length 4 with edges  $ \{1,2\}, \{2,3\}, \{3,4\}, \{4,5\}  $ and another one of length 3 with edges $  \{4,5\}, \{5,6\}, \{6,7\}  $  having two shared vertices $4 $ and $5$.  Vertices  $4$ and $5$ are connected by two distinct edges.
}\label{fig:dense_union_tree_center}
\end{figure}
\end{definition}

\begin{definition}[Unicyclic graphs]\label{def: TC}
    For $k \geq r\geq 3$, we define  $\mathcal{TC}^{k}_r$ to be the set of unicyclic graphs spanned by a closed walk of length $k$  containing exactly one cycle of length $r$, where and all other edges not in the cycle are visited exactly twice. See the left part of Figure~\ref{fig:dense_with_cycle}  for an example.
    
    We define the set $\mathcal{TC}^{k}$ as  $\mathcal{TC}^{k} \coloneqq \bigcup_{3\leq r\leq k} \mathcal{TC}^{k}_r$.
 Furthermore, for integers $k, h \geq 3$, define $\mathcal{TC}^{k,h}$ as the set of all graphs with two roots, formed by the union of a graph in $\mathcal{TC}^{k}_r$ and another graph in $\mathcal{TC}^{h}_r$ that share the same cycle of length $r$. 
    See  Figure~\ref{fig:dense_with_cycle}  for an example.
  \end{definition}
           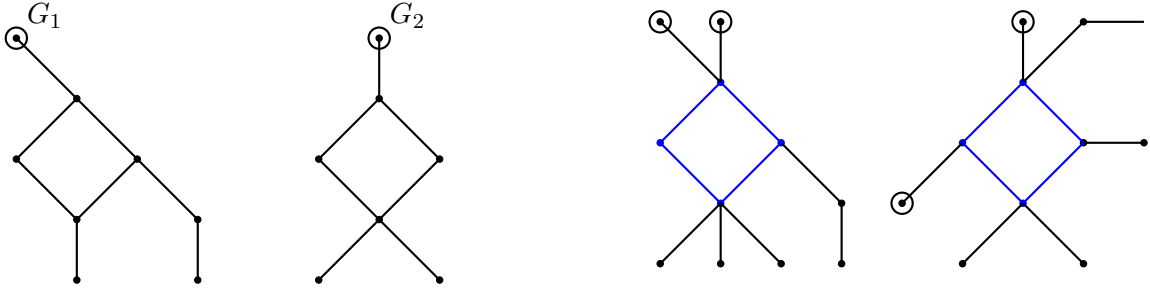
\begin{figure} 
     \begin{minipage}{0.46\textwidth}
    \centering
    \begin{tikzpicture}[scale=0.8]
        
        \filldraw[black] (1, 3) circle (1.5pt) ;
        \filldraw[black] (2, 1) circle (1.5pt) ;
        \filldraw[black] (2,2) circle (1.5pt) ;
        \filldraw[black] (2,4) circle (1.5pt) ;
        \filldraw[black] (1,5) circle (1.5pt);
        \draw[black, thick] (1,5) circle (5pt) node[above right] {$G_1$};
        \filldraw[black] (3,3) circle (1.5pt) ;
        \filldraw[black] (4,1) circle (1.5pt) ;
        \filldraw[black] (4,2) circle (1.5pt) ;

        \filldraw[black] (6, 3) circle (1.5pt) ;
        \filldraw[black] (6, 1) circle (1.5pt) ;
        \filldraw[black] (8, 1) circle (1.5pt) ;
        \filldraw[black] (7,2) circle (1.5pt) ;
        \filldraw[black] (7,4) circle (1.5pt) ;
        \filldraw[black] (7,5) circle (1.5pt);
        \draw[black, thick] (7,5) circle (5pt) node[above right] {$G_2$};
        \filldraw[black] (8,3) circle (1.5pt) ;

       \draw[thick] (1, 3) -- (2,4); 
        \draw[thick] (1, 3) -- (2,2);
       \draw[thick] (1,5) -- (2,4); 
       \draw[thick] (3,3) -- (4,2); 
       \draw[thick] (3,3) -- (2,4); 
       \draw[thick] (3,3) -- (2,2); 
       \draw[thick] (2, 1) -- (2,2); 
       \draw[thick] (4,2) -- (4,1); 

 \draw[thick] (7,5) -- (7,4); 
        \draw[thick] (7,4) -- (8,3);
       \draw[thick] (7,4) -- (6, 3);
         \draw[thick] (7,2) -- (8,3);
       \draw[thick](7,2) -- (6, 3); 
       \draw[thick] (7,2) -- (6, 1); 
       \draw[thick] (7,2) -- (8, 1); 

    \end{tikzpicture}
    \end{minipage}
    \hfill 
    \begin{minipage}{0.46\textwidth}
    \centering
    \begin{tikzpicture}[scale=0.8]

       \filldraw[blue] (1, 3) circle (1.5pt) ;
        \filldraw[black] (2, 1) circle (1.5pt) ;
        \filldraw[black] (1, 1) circle (1.5pt) ;
        \filldraw[black] (3, 1) circle (1.5pt) ;
        \filldraw[blue] (2,2) circle (1.5pt) ;
        \filldraw[blue] (2,4) circle (1.5pt) ;
        \filldraw[black] (1,5) circle (1.5pt);
        \draw[black, thick] (1,5) circle (5pt);
        \filldraw[black] (2,5) circle (1.5pt);
        \draw[black, thick] (2,5) circle (5pt);
        \filldraw[blue] (3,3) circle (1.5pt) ;
        \filldraw[black] (4,1) circle (1.5pt) ;
        \filldraw[black] (4,2) circle (1.5pt) ;

         \filldraw[blue] (6, 3) circle (1.5pt) ;
        \filldraw[black] (5, 2) circle (1.5pt) ;
         \draw[black, thick] (5, 2) circle (5pt);
        \filldraw[black] (6, 1) circle (1.5pt) ;
        \filldraw[black] (8, 1) circle (1.5pt) ;
       \filldraw[black] (8, 3) circle (1.5pt);
        \filldraw[blue] (7,2) circle (1.5pt) ;
        \filldraw[blue] (7,4) circle (1.5pt) ;
         \filldraw[black] (7,5) circle (1.5pt);
        \draw[black, thick] (7
        ,5) circle (5pt);
        \filldraw[black] (8,5) circle (1.5pt) ;
         \filldraw[black] (9,3) circle (1.5pt) ;

\draw[thick, blue] (1, 3) -- (2,4); 
        \draw[thick, blue] (1, 3) -- (2,2);
       \draw[thick] (1,5) -- (2,4);
       \draw[thick] (2,5) -- (2,4); 
       \draw[thick] (3,3) -- (4,2); 
       \draw[thick, blue] (3,3) -- (2,4); 
       \draw[thick, blue] (3,3) -- (2,2); 
       \draw[thick] (2, 1) -- (2,2); 
        \draw[thick] (2, 2) -- (1, 1);
         \draw[thick] (2, 2) -- (3, 1);
       \draw[thick] (4,2) -- (4,1); 

      \draw[thick] (7,5) -- (7,4); 
      \draw[thick] (8,5) -- (7,4);
      \draw[thick] (8,5) -- (9,5);
        \draw[thick, blue] (7,4) -- (8,3);
        \draw[thick] (8,3) --(9,3);
       \draw[thick, blue] (7,4) -- (6, 3); 
       \draw[thick] (6, 3) -- (5, 2); 
         \draw[thick, blue] (7,2) -- (8,3);
       \draw[thick, blue](7,2) -- (6, 3); 
       \draw[thick] (7,2) -- (6, 1); 
       \draw[thick] (7,2) -- (8, 1); 

    \end{tikzpicture}
    \end{minipage} \caption{ We given an example of $k = 12$, $h = 10$ and $r = 4$. The left figure are two graphs $ G_1 \in \mathcal{TC}^{k}_{r}$ and $G_2 \in \mathcal{TC}^{h}_{r} $. The right graph shows two possible graphs in $\mathcal{TC}^{k,h}$.}\label{fig:dense_with_cycle}
    \end{figure}

  With Definitions~\ref{def: graph_tree} and \ref{def: TC}, we are ready to state the CLT for $\overline{A}_n$ in the dense regime.

\begin{theorem}[CLT for $\overline{A}_n$, dense regime]\label{thm: centered_CLT_inhomo_dense}
  Let $A_n$  satisfy Assumption~\ref{assumption1} and $W'=W(1-pW)$. If $p$ is a constant  and Assumption~\ref{assumption:delta1} holds, then $\{X_{n,k}\}_{k\geq 2}$ converges to a centered Gaussian process $\{X_{k}\}_{k\geq 2}$ (in the sense of finite marginals) with
  \begin{small}
      
      \begin{align}
\mathrm{Cov}(X_k, X_h) =
    \begin{cases}
      p \sum_{G \in \mathcal {TC}^{k,h}} t( G, W')  &\text{if $k, h \geq 3$ are odd,} \\[10pt]
         \sum_{T \in \mathcal T^{k,h}_{1}}  t(T,W') -   4p \sum_{T \in \mathcal T^{k,h}_{2}}  t(T,W') +   p \sum_{G \in \mathcal{TC}^{k,h}} t(  G, W') 
        &\text{if $k,h \geq 2$ are even,} \\[10pt]
        0  &\text{otherwise.}
    \end{cases}
\end{align}
  \end{small}
\end{theorem}

When $p\to 0$ and $np\to\infty$, a different CLT holds under a weaker assumption on the graphon convergence of $W_n$, as shown in the next theorem.

\begin{theorem}[CLT for $\overline{A}_n$, sparse regime] \label{thm: centered_CLT_inhomo_sparse1}
If $p\to 0$, $np\to\infty$ as $n\to\infty $ and Assumption~\ref{assumption:tree} holds, then $\{X_{n,k}\}_{k\geq 2}$ converges to a centered Gaussian process $\{X_{k}\}_{k\geq 2}$ with
 \begin{align}\label{eq:Cov_CLT_sparse_centered}
\mathrm{Cov}(X_k, X_h) =
    \begin{cases}
        \sum_{T \in \mathcal 
 T^{k,h}_{1}}  \beta_T  &\text{if $k,h$ are even,} \\ 
        0  &\text{otherwise.}
    \end{cases}
\end{align}
\end{theorem}

 Appendix~\ref{sec:size} includes  calculations of the sizes of $\mathcal T_{1}^{k,h}, \mathcal T_{2}^{k,h}$, and $\mathcal{TC}^{k,h}$, which help  interpret Theorems~\ref{thm: centered_CLT_inhomo_dense} and \ref{thm: centered_CLT_inhomo_sparse1} in the homogeneous case  $s_{ij}=1$ for all $i,j\in [n]$.

In Theorem~\ref{thm: centered_CLT_inhomo_sparse1}, the only nontrivial covariance term arises from gluing two rooted planar trees by sharing a single overlapping edge. For the bounded expected degree regime $np\to c\in (0,\infty)$, a more general tree gluing is required.

  \begin{definition}[Tree gluing]\label{def:glue}
    For any two rooted planar trees $T_1, T_2$, we denote \( \mathcal{P}_{\#}(T_1, T_2) \)  the set of all trees formed by identifying subsets of edges in  $T_1$ and  $T_2$ such that the union remains a tree.  See Figure~\ref{fig: tree_gluing}  for  examples.
    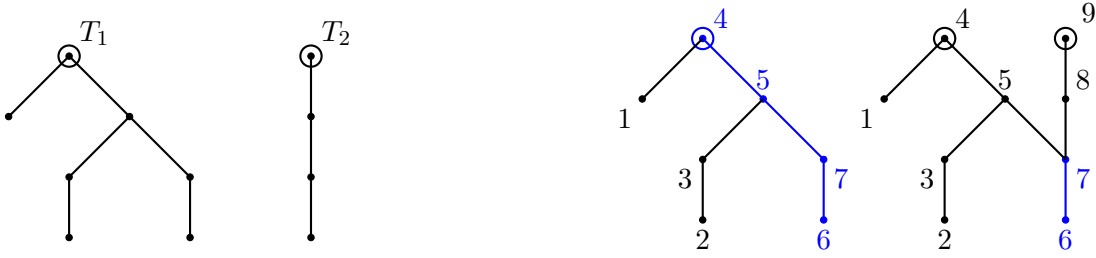
\begin{figure} 
     \begin{minipage}{0.46\textwidth}
    \centering
    \begin{tikzpicture}[scale=0.8]
        \filldraw[black] (1, 3) circle (1.5pt) ;
        \filldraw[black] (2, 1) circle (1.5pt) ;
        \filldraw[black] (2,2) circle (1.5pt) ;
        \filldraw[black] (2,4) circle (1.5pt) ;
        \draw[black, thick]  (2,4) circle (5 pt)  node[above right] {$T_1$};
        \filldraw[black] (3,3) circle (1.5pt) ;
        \filldraw[black] (4,1) circle (1.5pt) ;
        \filldraw[black] (4,2) circle (1.5pt) ;

        \filldraw[black] (6,1) circle (1.5pt) ;
        \filldraw[black] (6,2) circle (1.5pt) ;
        \filldraw[black] (6,3) circle (1.5pt) ;
        \filldraw[black] (6,4) circle (1.5pt) ;
         \draw[black, thick] (6,4) circle (5pt) node[above right] {$T_2$} ;

       \draw[thick] (1, 3) -- (2,4); 
       \draw[thick] (3,3) -- (4,2); 
       \draw[thick] (3,3) -- (2,4); 
       \draw[thick] (3,3) -- (2,2); 
       \draw[thick] (2, 1) -- (2,2); 
       \draw[thick] (4,2) -- (4,1); 

    \draw[thick] (6,1) -- (6,2);
    \draw[thick] (6,2) -- (6,3);
    \draw[thick] (6,4) -- (6,3);
 
    \end{tikzpicture}
    \end{minipage}
    \hfill 
    \begin{minipage}{0.46\textwidth}
    \centering
    \begin{tikzpicture}[scale=0.8]
        
        \filldraw[black] (1, 3) circle (1.5pt) node[below left] {1};
        \filldraw[black] (2, 1) circle (1.5pt) node[below] {2};
        \filldraw[black] (2,2) circle (1.5pt) node[below left] {3};
        \filldraw[black] (2,4) circle (1.5pt) ;
         \draw[black, thick] (2,4) circle (5pt)
        node[above right] {4};
        \filldraw[black] (3,3) circle (1.5pt) node[above] {5};
        \filldraw[blue] (4,1) circle (1.5pt) node[below] {6};
        \filldraw[blue] (4,2) circle (1.5pt) node[below right] {7};
        \filldraw[black] (4,3) circle (1.5pt) node[above right] {8};
        \filldraw[black] (4,4) circle (1.5pt) ;
        \draw[black, thick] (4,4) circle (5pt) node[above right=2pt] {9};

     \draw[thick] (3,3) -- (2,4);
       \draw[thick] (1, 3) -- (2,4);
       \draw[thick] (3,3) -- (4,2);
       \draw[thick] (3,3) -- (2,2);
       \draw[thick] (2, 1) -- (2,2);
       \draw[thick, blue] (4,2) -- (4,1);
       \draw[thick] (4,2) -- (4,3);
       \draw[thick] (4,3) -- (4,4);

         \filldraw[black] (-3, 3) circle (1.5pt) node[below left] {1};
        \filldraw[black] (-2, 1) circle (1.5pt) node[below] {2};
        \filldraw[black] (-2,2) circle (1.5pt) node[below left] {3};
        \filldraw[blue] (-2,4) circle (1.5pt) ;
        \draw[blue, thick] (-2,4) circle (5pt)
        node[above right] {4};
        \filldraw[blue] (-1,3) circle (1.5pt) node[above] {5};
        \filldraw[blue] (0,1) circle (1.5pt) node[below] {6};
        \filldraw[blue] (0,2) circle (1.5pt) node[below right] {7};

       \draw[thick] (-3, 3) -- (-2,4); 
       \draw[thick, blue] (-1,3) -- (0,2); 
       \draw[thick, blue] (-1,3) -- (-2,4); 
       \draw[thick] (-1,3) -- (-2,2); 
       \draw[thick] (-2, 1) -- (-2,2); 
       \draw[thick, blue] (0,2) -- (0,1); 

    \end{tikzpicture}
    \end{minipage}
    \caption{Given the rooted trees \( T_1 \) and \( T_2 \) in the left graph, the right graph shows two examples in \( P_{\#}(T_1, T_2) \), where \( T_1 \) and \( T_2 \) overlap in the blue part for each case.}\label{fig: tree_gluing}
\end{figure}

\end{definition}  
\begin{theorem}[CLT for $\overline{A}_n$, bounded expected degree regime] \label{thm: inhomo_bounded_c}
Let $A_n$  satisfy Assumption~\ref{assumption1}. 
Suppose $np\to  c\in (0,\infty)$ as $n \to \infty$ and Assumption~\ref{assumption:tree} holds. Then $\{X_{n,k}\}_{k\geq 2}$ converges to a centered Gaussian process $\{X_k\}_{k\geq 2}$ with
\begin{align}\label{equ: bounded_nc_cov_thm}
    \mathrm{Cov}(X_k, X_h) = \begin{cases}
     \sum_{i\leq k/2, j\leq h/2}\sum_{T_1\in \mathcal T_i, T_2\in \mathcal T_j} ~\sum_{T \in \mathcal{P}_{\#}(T_1, T_2)} \frac{\beta_T}{c^{(k+h)/2-v(T)}}  & k,h \text{ are even},\\
    0 &\text{otherwise}
    \end{cases}
\end{align}
where  $\mathcal T_i$  is the set of all rooted planar trees with $i$  edges.
\end{theorem}

For the homogeneous case ($s_{ij}=1$ for all $i,j\in [n]$), Theorem~\ref{thm: inhomo_bounded_c} recovers \cite[Theorem 1.4 and Theorem 2.2]{benaych2014central} for the adjacency matrix $A_n$ of $G(n,p)$ with $p=\frac{c}{n}$.
As a corollary of Theorem~\ref{thm: inhomo_bounded_c}, we can generalize \cite[Theorem 1.4]{benaych2014central} to sparse generalized Wigner matrices.

\begin{cor}[CLT for inhomogeneous random graphs with constant expected degree]\label{cor:CLT}
 Assume  for all $x\in [0,1]$, 
\begin{align}\label{eq:constant_graphon_degree}
\lim_{n\to\infty} \int_{0}^1 W_n(x,y) dy=1.
\end{align}
 Then  $\{X_{n,k}\}_{k\geq 2}$ converges to a centered Gaussian process $\{\widetilde{L}_k\}_{k\geq 2}$ with
\begin{align}
    \mathrm{Cov}(X_k, X_h) = \sum_{i\leq k/2, j\leq h/2}\sum_{T_1\in \mathcal T_i, T_2\in \mathcal T_j} ~\sum_{T \in \mathcal{P}_{\#}(T_1, T_2)} \frac{1}{c^{(k+h)/2-v(T)}}.
\end{align}
\end{cor}
\begin{proof}
    From \cite[Lemma 4.1]{zhu2020graphon}, under the assumption~\eqref{eq:constant_graphon_degree}, $\beta_T=1$ for all finite trees. Then the statement follows from \eqref{equ: bounded_nc_cov_thm}.
\end{proof}

\subsection{Adjacency matrices}\label{sec:non_centered}
We define the rescaled linear spectral statistics of $A_n$ as 
\begin{align}\label{eq:Lnk}
L_{n,k} = np^{\frac{1}{2}+\mathbf{1}\{k=2\}} \cdot \left( \mathrm{tr}\left(\frac{A_n}{np}\right)^k - \mathbb{E}\mathrm{tr}\left(\frac{A_n}{np}\right)^k \right).
\end{align}

\begin{remark}[Scaling factor for $k=2$]\label{rmk:k=2}
For $k=2$, an additional scaling factor $p$ is included to ensure $\mathrm{Var}(L_{n,2})=\Theta(1)$ when $np=n^{\Omega(1)}$, as noticed in \cite{diaconu2022two}. Specifically, $\mathrm{tr}(A_n^2)=\sum_{ij} a_{ij}$ is a sum of independent Bernoulli random variables, and $L_{n,2}$ satisfies the classical CLT for independent random variables. For $k\geq 3$, the dependence among eigenvalues results in a different scaling factor.
\end{remark}

To describe the limiting variance, we introduce the following graph classes:

\begin{definition}[Two cycles with overlapping edges]\label{def: F-graph}
For any $k,h\geq 3$,  we define $ F_1^{k,h} = (V, E)$ as a simple graph with $|V| = |E|-1 = (k+h) - 2$ constructed from two cycles of length $k$ and $h$ sharing one overlapping edge.  See the left part of Figure~\ref{fig:dense_union_graph} for an example. 

Similarly,  we define ${F}_2^{k,h} = (V, E)$ as a multigraph with $|V|  = |E|-2 = (k+h)-2$ constructed by two cycles of length $k$ and $h$, sharing two vertices connected by two distinct edges.  See the right part of Figure~\ref{fig:dense_union_graph} for an example. 
\end{definition}
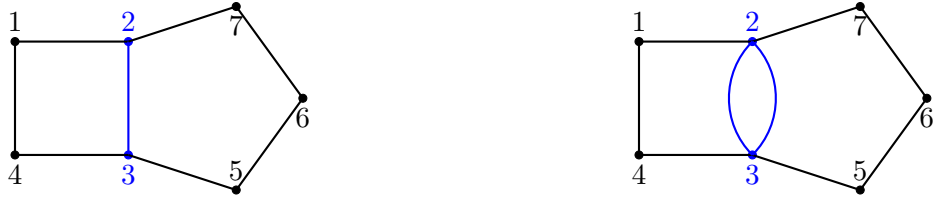
\begin{figure} 
    \centering
    \begin{minipage}{0.5\textwidth}
    \centering
    \begin{tikzpicture}[scale=1.5]
       \filldraw[black] (0, 1) circle (1pt) node[above] {1};
        \filldraw[blue] (1, 1) circle (1pt) node[above] {2};
        \filldraw[black] (1.951,-0.3091) circle (1pt) node[above] {5};
        \filldraw[black] (0, 0) circle (1pt) node[below] {4};
        \filldraw[blue] (1, 0) circle (1pt) node[below] {3};
        \filldraw[black] (2.538841, 0.5) circle (1pt) node[below] {6};
        \filldraw[black] (1.951, 1.3091) circle (1pt) node[below] {7};

        \draw[thick] (0,1) -- (1,1); 
        \draw[thick, blue] (1,1) -- (1,0); 
        \draw[thick] (1,0) -- (0,0);  
        \draw[thick] (0,0) -- (0,1);  
        \draw[thick] (1, 1) -- (1.951, 1.3091); 
        \draw[thick] (1, 0) -- (1.951,-0.3091);
        \draw[thick] (1.951, 1.3091) -- (2.538841, 0.5);
        \draw[thick] (1.951,-0.3091) -- (2.538841, 0.5);
    \end{tikzpicture}
    \end{minipage}%
    \hfill 
    \begin{minipage}{0.5\textwidth}
    \centering
    \begin{tikzpicture}[scale=1.5]
        \filldraw[black] (0, 1) circle (1pt) node[above] {1};
        \filldraw[blue] (1, 1) circle (1pt) node[above] {2};
        \filldraw[black] (1.951,-0.3091) circle (1pt) node[above] {5};
        \filldraw[black] (0, 0) circle (1pt) node[below] {4};
        \filldraw[blue] (1, 0) circle (1pt) node[below] {3};
        \filldraw[black] (2.538841, 0.5) circle (1pt) node[below] {6};
        \filldraw[black] (1.951, 1.3091) circle (1pt) node[below] {7};

        \draw[thick] (0,1) -- (1,1);  
        \draw[thick] (1,0) -- (0,0);  
        \draw[thick] (0,0) -- (0,1);  
        \draw[thick] (1, 1) -- (1.951, 1.3091); 
        \draw[thick] (1, 0) -- (1.951,-0.3091);
        \draw[thick] (1.951, 1.3091) -- (2.538841, 0.5);
        \draw[thick] (1.951,-0.3091) -- (2.538841, 0.5);

        \draw[thick,blue] (1, 1) to[out=225, in=135] (1, 0);
        \draw[thick,blue] (1, 1) to[out=315, in=45] (1, 0);  
    \end{tikzpicture}
    \end{minipage}
    \caption{ We give an example of $k=4$ and $h=5$. The left figure is an example of $F^{k,h}_1$ constructed by cycle one of length 4 with $E = \{\{1,2\}, \{2,3\},\{3,4\}, \{4,1\} \}$ and cycle two of length 5 with $ E = \{ \{2,3\}, \{2,7\}, \{7,6\}, \{6,5\}, \{5,3\}\}$ having one overlapping edge $ \{2,3\}$.  The right figure is an example of $F^{k,h}_2$  constructed by cycle one of length 4 with $E = \{ \{1,2\}, \{2,3\},\{3,4\}, \{4,1\} \}$ and cycle two of length 5 with $ E = \{ \{2,3\}', \{2,7\}, \{7,6\}, \{6,5\}, \{5,3\} \}$ having two shared vertices 2 and 3. And vertices 2 and 3 are connected by two distinct edges}
    \label{fig:dense_union_graph}
\end{figure}

\begin{definition}[A cycle with one multi-edge]\label{def:multi_cycle}
    Let $K_2$ be a complete graph on two vertices, and $C_2$ be a multigraph on 2 vertices with 2 multi-edges. For $h\geq 3$, $C_h$ is a cycle of length $h$ and $C_{2,h}$ is a cycle of length $h$ with one multi-edge. See Figure~\ref{fig:cycle_mult_edge} for an example.
\end{definition}

\begin{figure} 
    \centering
    \begin{minipage}{0.5\textwidth}
    \centering
    \begin{tikzpicture}[scale=1.5]
        \filldraw[blue] (1, 1) circle (1pt) node[above] {2};
        \filldraw[black] (1.951,-0.3091) circle (1pt) node[above] {5};
        \filldraw[blue] (1, 0) circle (1pt) node[below] {3};
        \filldraw[black] (2.538841, 0.5) circle (1pt) node[below] {6};
        \filldraw[black] (1.951, 1.3091) circle (1pt) node[below] {7};

        \draw[thick, blue] (1,1) -- (1,0); 
        \draw[thick] (1, 1) -- (1.951, 1.3091); 
        \draw[thick] (1, 0) -- (1.951,-0.3091);
        \draw[thick] (1.951, 1.3091) -- (2.538841, 0.5);
        \draw[thick] (1.951,-0.3091) -- (2.538841, 0.5);
    \end{tikzpicture}
    \end{minipage}%
    \hfill 
    \begin{minipage}{0.5\textwidth}
    \centering
    \begin{tikzpicture}[scale=1.5]
        \filldraw[blue] (1, 1) circle (1pt) node[above] {2};
        \filldraw[black] (1.951,-0.3091) circle (1pt) node[above] {5};
        \filldraw[blue] (1, 0) circle (1pt) node[below] {3};
        \filldraw[black] (2.538841, 0.5) circle (1pt) node[below] {6};
        \filldraw[black] (1.951, 1.3091) circle (1pt) node[below] {7};

        \draw[thick] (1, 1) -- (1.951, 1.3091); 
        \draw[thick] (1, 0) -- (1.951,-0.3091);
        \draw[thick] (1.951, 1.3091) -- (2.538841, 0.5);
        \draw[thick] (1.951,-0.3091) -- (2.538841, 0.5);

        \draw[thick,blue] (1, 1) to[out=225, in=135] (1, 0);
        \draw[thick,blue] (1, 1) to[out=315, in=45] (1, 0);  
    \end{tikzpicture}
    \end{minipage}
    \caption{We give an example of $h=5$. The left figure is an example of $C_h$ and the right figure is an example of $C_{2,h}$. In particular, the blue part in the left figure is an example of $K_2$ and the blue part in the right figure is an example of $C_2$.}
    \label{fig:cycle_mult_edge}
\end{figure}
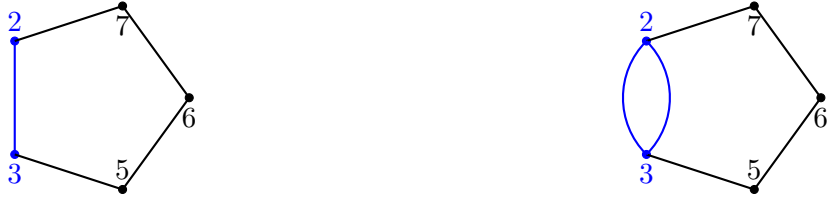

The first result is a CLT for the linear spectral statistics of $A_n$ in the dense regime, where we use Definitions~\ref{def: F-graph} and~\ref{def:multi_cycle} to describe the limiting covariance.
\begin{theorem}[CLT for $A_n$, dense regime]\label{thm:nc_CLT_inhomo_dense}
      Let $A_n$  satisfy Assumption~\ref{assumption1},   and suppose Assumption~\ref{assumption:delta1} holds. Then
  $\{L_{n,k}\}_{k\geq 2}$ converges in distribution   to a centered Gaussian process $\{L_k\}_{k\geq 2}$ with
  \begin{align}
      \mathrm{Cov}(L_k, L_h)= \begin{cases} 
    2\left( t(K_2,W)-p \cdot t(C_2,W) \right)  & k=h=2\\
       2h \left( t(C_h, W)-p\cdot t(C_{2,h},W)\right)    & k=2,h\geq 3\\
     2kh \left( t(F_1^{k,h}, W) - p \cdot t(F_2^{k,h}, W) \right)     &    k,h\geq 3
      \end{cases}.
  \end{align}
  \end{theorem}
Compared to Theorem~\ref{thm: centered_CLT_inhomo_dense}, the covariance structure in Theorem~\ref{thm:nc_CLT_inhomo_dense}
is defined by cycles or a union of cycles, and the scaling factor differs.

Now we turn to the sparse regime where $np\to\infty, p\to 0$. Surprisingly, there are infinitely many phase transitions for $p$ decreasing at different rates. In addition, a change of scaling appears from $np=n^{\Omega(1)}$ to $np=n^{o(1)}$. To the best of our knowledge, this is a new phase transition for linear spectral statistics in the literature for sparse random matrices. 
When $np=n^{o(1)}$, define the rescaled  statistics
\begin{align}\label{eq:newlimit}
\widetilde{L}_{n,k} = \sqrt{p} \left( \mathrm{tr}\left(\frac{A_n}{\sqrt{np}}\right)^k - \mathbb{E}\mathrm{tr}\left(\frac{A_n}{\sqrt{np}}\right)^k \right).
\end{align}

  \begin{theorem}[CLT for $A_n$, sparse  regime]\label{thm:nc_CLT_inhomo_sparse}
  Let $A_n$ satisfy Assumption~\ref{assumption1}.
  \begin{enumerate}
      \item   Assume $np\gg n^{1/2}$ and $p\to 0$.
 Suppose Assumption~\ref{assumption:deltaBox} holds, then
  $\{L_{n,k}\}_{k\geq 2}$ converges to a centered Gaussian process $\{L_k\}_{k\geq 2}$ with
  \begin{align}\label{equ: dense_nc_cov_thm}
      \mathrm{Cov}(L_k, L_h)= \begin{cases} 
    2 \cdot t(K_2,W)  & k=h=2\\
       2h \cdot  t(C_h, W)    & k=2,h\geq 3\\
     2kh \cdot t(F_1^{k,h}, W)      &    k,h\geq 3
      \end{cases}, 
  \end{align}
  where $F_1^{k,h}, K_2, C_{h}$ are defined in Definitions~\ref{def: F-graph} and~\ref{def:multi_cycle}.
  \item Assume $n^{1/2}p\to c$ for some $c\in (0,\infty)$.  Suppose Assumption~\ref{assumption:deltaBox} holds, then
  $\{L_{n,k}\}_{k\geq 2}$ converges to a centered Gaussian process $\{L_k\}_{k\geq 2}$ with
  \begin{align}
      \mathrm{Cov}(L_{k}, L_{h})= \begin{cases} 
    2 \cdot t(K_2,W) & k=h=2\\[4pt]
       2h \cdot  t(C_h, W)   & k=2,h= 3 \text{ or } h \geq 5\\[4pt]
     8 \cdot  t(C_4, W) + 4 c^{-2} \sum_{T\in \mathcal T_2} t(T, W)  & k=2,h = 4 \\[4pt]
       6 c^{-2} \cdot t(C_3,W) + 18 \cdot t(F_1^{3,3}, W)  &    k=h = 3 \\[4pt] 
    c^{-4}\sum_{T \in \mathcal  T^{4,4}_{1}}  t(T,W) + 32 \cdot t(F_1^{4,4}, W)  &    k=h = 4 \\[4pt]
       2kh \cdot t(F_1^{k,h}, W)   &  \text{otherwise} 
      \end{cases}.
  \end{align}
  \item Assume $n^{1/m}\ll np \ll n^{1/{(m-1)}}$ for an integer $m\geq 3$.  Suppose Assumption~\ref{assumption:deltaBox} holds, then
  $\{L_{n,k}\}$ for $k=2$ and $k\geq 2m-1$ converges to a centered Gaussian process $\{L_k\}$ such that 
  \begin{align}
  \mathrm{Cov}(L_k, L_h)= \begin{cases} 
    2 \cdot t(K_2,W)  & k=h=2\\
       2h \cdot  t(C_h, W)    & k=2,h\geq 2m-1\\
     2kh \cdot t(F_1^{k,h}, W)    &    k,h\geq 2m-1
      \end{cases}.
      \end{align}
  \item Assume $ n^{1-\frac{1}{m}}p\to c$ for some constant $c\in (0,\infty)$.  Suppose Assumption~\ref{assumption:deltaBox} holds, then
  $\{L_{n,k}\}$ with $k=2$ and $k\geq 2m-1$ converges to a centered Gaussian process $\{L_k\}$ such that
  \begin{small}
      \begin{align}
    \mathrm{Cov}(L_{k}, L_{h})= \begin{cases} 
    2 \cdot t(K_2,W)  & k=h=2 \\[4pt]
    2h \cdot  t(C_h, W)    & k=2,h =  2m-1 \text{ or } h \geq 2m+1 \\[4pt]
2m \cdot c^{-m} \sum_{T\in \mathcal T_{m}}t(T,W)+ 4m \cdot  t(C_{2m}, W)   & k=2, h = 2m\\[4pt]
       c^{-2m} \sum_{T \in \mathcal  T^{2m,2m}_{1}}  t(T,W) +8m^2 \cdot t(F_1^{2m,2m}, W)  & k = h = 2m \\[4pt]
     2kh \cdot t(F_1^{k,h}, W)    &   \text{otherwise} 
      \end{cases}.
\end{align}
  \end{small}

  \item      Assume $np=n^{o(1)}$ and $np\to\infty$. Suppose Assumption~\ref{assumption:tree} holds, then
  $\{\widetilde{L}_{n,k}\}_{k\geq 2}$ converges to a centered Gaussian process $\{L_k\}_{k\geq 2}$ with
 the same covariance as in \eqref{eq:Cov_CLT_sparse_centered}.
  \end{enumerate}
  \end{theorem}

\begin{remark}[Phase transitions at $m\geq 3, m\in \mathbb Z$]\label{rmk:scaling}
    In Parts (3) and (4) of Theorem~\ref{thm:nc_CLT_inhomo_sparse}, we do not include the case for $3\leq k\leq 2m-2$. This is because in the proof of Theorem~\ref{thm:nc_CLT_inhomo_sparse}, one can see that under the conditions in Parts (3) and (4) Theorem~\ref{thm:nc_CLT_inhomo_sparse} when $n^{1/m}\ll np \ll n^{1/{(m-1)}}$, $\mathrm{Var}(L_{n,k})\to\infty$   at different scales for different  $3\leq k\leq 2m-2$. However, for all fixed $k\geq 2m-1$, $\mathrm{Var}(L_{n,k})=O(1)$ and we can compute the limiting covariance of the Gaussian process precisely. Therefore, besides the special case $k=2$ as explained in Remark \ref{rmk:k=2}, $k\geq 2m-1$ is the necessary and sufficient conditions to obtain a CLT in the regime when  $n^{1/m}\ll np \ll n^{1/{(m-1)}}$ for $L_{n,k}$.
\end{remark}

Finally, for the bounded expected regime, the same CLT  for $\overline{A}_n$ also holds for $A_n$:
\begin{theorem}[CLT for $A_n$, bounded expected degree regime]\label{thm:nc_CLT_inhomo_bounded}
Assume $np\to  c\in (0,\infty)$ as $n \to \infty$.
Suppose for any finite tree T, $t(T, W_n)\to \beta_T$ as $n \rightarrow \infty$. Then
  $\{\widetilde{L}_{n,k}\}_{k\geq 2}$ converges to a centered Gaussian process $\{L_k\}_{k\geq 2}$ with the same covariance as in \eqref{equ: bounded_nc_cov_thm}.
\end{theorem}

\begin{remark}[Comparison with CLTs for $\overline{A}_n$ in Section~\ref{sec:centered}]
    We see from Theorem~\ref{thm:nc_CLT_inhomo_dense} and Parts (1)-(4) in Theroem~\ref{thm:nc_CLT_inhomo_sparse}, when $np=n^{\Omega(1)}$, the CLTs for linear spectral statistics are different between $A_n$ and $\overline{A}_n$. However, when $np=n^{o(1)}$, Part (5) in Theorem~\ref{thm:nc_CLT_inhomo_sparse} and \ref{thm:nc_CLT_inhomo_bounded} show that same limiting Gaussian process holds for both $A_n$ and $\overline{A}_n$, revealing a phase transition in the centering effect on linear spectral statistics.
\end{remark}

\begin{remark}[Comparison with \cite{diaconu2022two}]
To the best of our knowledge, the only CLT for the (non-centered) adjacency matrix is given in \cite{diaconu2022two} for the homogeneous Erd\H{o}s-R\'{e}nyi graph $G(n,p)$. \cite[Theorem 1]{diaconu2022two} showed that if $np(1-p)\to\infty$, for $\frac{\log n}{c_0\log(np(1-p))}\leq k\leq c_0 (np(1-p)))^{1/4}$ and $k$ is even, the following convergence in distribution holds:
\begin{align}\label{eq:diaconu}
    \frac{\mathrm{tr}(A_n^{k})-\mathbb E[\mathrm{tr}(A_n^{k})]}{k(np)^{k-1}\sqrt{p(1-p)}}\to N(0,2).
\end{align}
In addition, Equation \eqref{eq:diaconu} is used as an ingredient to study the fluctuation of the largest eigenvalue of $A_n$ in \cite{diaconu2022two}.
We compare our Theorems~\ref{thm:nc_CLT_inhomo_dense}, \ref{thm:nc_CLT_inhomo_sparse}, \ref{thm:nc_CLT_inhomo_bounded}  with \cite{diaconu2022two} under the $G(n,p)$ model. In the dense case, Theorem~\ref{thm:nc_CLT_inhomo_dense} recovers \eqref{eq:diaconu} for $k\geq 3$ and $k$ is fixed, while \cite{diaconu2022two} requires $k\geq C_0$  and $k$ is even for some constant $C_0>0$ but allows $k$ grows with $n$. In the sparse regime where $np\gg n^{1/m}$ or $np=cn^{1/m}$ for $m\geq 2$, Theorem~\ref{thm:nc_CLT_inhomo_sparse} obtain \eqref{eq:diaconu} for $k\geq 2m-1$, while \cite{diaconu2022two} requires $k\geq C_0 m$ and $k$ is even. When $np=n^{o(1)}$, \cite[Theorem 1]{diaconu2022two} holds for $k=\omega(\log n)$ and $k$ is even, while our CLT in part (5) of Theorem~\ref{thm:nc_CLT_inhomo_sparse} holds for fixed $k\geq 2$ under a different scaling \eqref{eq:newlimit}. This shows the difference between fixed $k$ and growing $k$ in this regime.
When $np\to c$, our Theorem~\ref{thm:nc_CLT_inhomo_bounded} gives a CLT under the different scaling \eqref{eq:newlimit} while the result in \cite{diaconu2022two} does not apply.
\end{remark}

\section{Proof for the centered adjacency matrices}\label{sec:proof_centered}

\subsection{Proof of Theorem~\ref{thm: centered_CLT_inhomo_dense}}

The proof consists of two steps.  In this proof, we will also introduce several definitions that will be repeatedly used in other theorems.
Let $S_n'=(s'_{ij})$ be another variance profile matrix related to $S_n$, given by 
    $s'_{ij}=s_{ij} (1-ps_{ij})$.
We associate a graphon $W'_n$ to the matrix $S'_n$.

  \begin{definition}[$k$-tuplet]\label{def: k-tuplet_of_numbers1}
    We define a \( k \)-tuplet  as \( I_k = (i_1, i_2, \dots, i_k) \), where each \( i_j \) is an element of the set \( \{1, 2, \dots, n\} \). We define $\overline{a}_{I_k}=\overline{a}_{i_1 i_2} \overline{a}_{i_2 i_3} \dots \overline{a}_{i_k i_1}$.
\end{definition}

 \noindent     \textbf{Step 1: Computing the limiting covariance.}\label{step: step1_dense_center} 
  We start with an asymptotic analysis of the covariance of $X_{n,k}, k\geq 2$ such that  \begin{align}\label{equ: dense_c_cov1}
        \mathrm{Cov}( X_{n,k}, X_{n,h}) &= \frac{p}{(np)^{(k+h)/2}} \sum_{I_k,J_h} \mathrm{Cov}(\overline{a}_{I_k}, \overline{a}_{J_h})\\
        &= \frac{p}{(np)^{(k+h)/2}} \sum_{I_k,J_h} \mathbb{E}[\overline{a}_{I_k} \overline{a}_{J_h} ] - \mathbb{E}[\overline{a}_{I_k}  ] \mathbb{E} [\overline{a}_{J_h} ],
    \end{align}
    where the sum is over all $k$-tuplet $i_1,\dots,i_k$ and $h$-tuplet $j_1,\dots,j_h$.

\begin{definition}[Union graph]\label{def: union_graph}
    To each pair of $\overline{a}_{I_k}, \overline{a}_{J_h}$,  we associate a graph \( G_{I_k} \) to \( a_{I_k} \), constructed over the set of distinct indices  appearing in $I_k$ and edges formed by the unordered pairs $\{i_j,i_{j+1}\}$. Then $a_{I_k}$ is a closed walk on $G_{I_k}$ through all vertices. Similarly, We construct $ G_{J_h}$ associated to \( a_{J_h} \). We denote the union graph of $G_{I_k}$ and $ G_{J_h}$ by $G_{I_k\cup J_h}$.
\end{definition}

Then to find the asymptotic value of \eqref{equ: dense_c_cov1}, it suffices to estimate the contribution from all possible union graphs $G_{I_k\cup J_h}$.
    For a given $G_{I_k\cup J_h}$, since $p$ is a constant and with Assumption \ref{assump: dense_1}, when $ v < (k+h)/2$, we have the contribution from all $G_{I_k\cup J_h}$ with a given vertex number $v$ and edge number $e$ is bounded by \begin{align}
        \frac{p}{(np)^{(k+h)/2}} \cdot n^v \left(Cp\right)^e = O( n^{v-(k+h)/2}) = o(1).
    \end{align}
    Therefore, it remains to consider all union graphs with \begin{align}\label{equ: dense_c_cov_ine_11}
        v-(k+h)/2 \geq 0.
    \end{align}

    Since each entry in $A_n$ is independent, to have a nonzero $\mathrm{Cov}(\overline{a}_{I_k}, \overline{a}_{J_h})$,  $G_{I_k \cup J_h}$ must be  connected. Also, by Assumption \ref{assumption1} that $a_{ii}=0, i\in [n]$, $G_{I_k \cup J_h}$ cannot contain any loop. Let $e_{I_k}, e_{J_h}$ be the number of distinct edges in $G_{I_k},G_{j_h}$, respectively. 
These restrictions lead to 
  \begin{align}
(k+h)-1\geq  e_{I_k} + e_{J_h} -1 \geq e.
 \end{align}

In addition, since each entry is centered, to have a nontrivial $\mathrm{Cov}(\overline{a}_{I_k}, \overline{a}_{J_h})$ in the limit, each edge in  \( G_{I_k \cup J_h} \) must be visited at least twice in the walk defined by \( I_k \) and \( J_h \). These conditions lead to \begin{align}\label{equ: dense_c_cov_ine_22}
    v - 1 \leq e \leq (k + h)/2.
\end{align} 

We first consider when $k+h$ is odd. Since $e$ and $v$ are integers, \eqref{equ: dense_c_cov_ine_11} and \eqref{equ: dense_c_cov_ine_22} together can be further restricted to \begin{align}\label{equ: dense_c_kh_odd_cond}
    (k+h)/2 - \frac{1}{2}  \leq v-1 \leq e \leq (k+h)/2 - \frac{1}{2}.
\end{align}
This implies the only case we have is $ v-1 = e = (k+h)/2 - 1/2 $ and $G_{I_k \cup J_h} $ is a tree. In the closed walk defined by $I_k \cup J_h $,  each edge is walked twice, except one is walked three times. This scenario is not feasible since an odd number of walking times is not possible in a closed walk on a rooted tree. Then, when $k+h$ is odd, the limiting covariance is 0. 

Thus, we only need to consider when $k+h$ is even. With \eqref{equ: dense_c_cov_ine_11} and \eqref{equ: dense_c_cov_ine_22},  the union graphs $G_{I_k \cup J_h}$ that contribute to the limiting variance must satisfy  \begin{align}
 e \leq v \leq e+1.
 \end{align}  
Therefore $v=e$ or $e+1$. We then consider 2 cases below.

\smallskip 
 \textbf{Case 1: \( v = e+1 \).}\label{case: dense_center_case1}
 In this case, $G_{I_k \cup J_h} $ is a tree formed by the union of two smaller tres, $G_{I_k}$ and $G_{J_h}$ with overlapping edges, which implies $  e_{I_k} + e_{J_h} -1 \geq e$. At the same time,  each \( \overline{a}_{I_k}, \overline{a}_{J_h} \) represents a closed walk on a tree that requires every edge to be visited at least twice, which leads to $e_{I_k}\leq \frac{k}{2},  e_{J_h}\leq \frac{h}{2}$, and $k,h$ are even. Together with \eqref{equ: dense_c_cov_ine_11}, we have \begin{align}\label{equ: dense_c_case1}
    (k+h)/2-1 \geq  e_{I_k} + e_{J_h} -1 \geq e \geq (k+h)/2-1.
 \end{align}
  Then it implies $k,h$ are even and $ e = v-1 = (k+h)/2 -1 $. Thus, the union is a planar tree formed by two rooted planar trees with an overlapping edge, which belongs to $ \mathcal T^{k,h}_1$ defined in Definition~\ref{def: graph_tree}.

\textbf{Case 2: \(  v =e \).}
In this case, $G_{I_k \cup J_h}$ is a graph with  one cycle. From  \eqref{equ: dense_c_cov_ine_11}, we have $ e \geq (k+h)/2$. From  \eqref{equ: dense_c_cov_ine_22}, we have $ e \leq (k+h)/2$. Hence, $ v = e = (k+h)/2$, which implies each edge in the union graph is visited exactly twice by $I_k,J_h$, and it holds only when $k,h \geq 3$ since a cycle has length at least $3$. In this case, both $G_{I_k}$ and $G_{J_h}$ are a rooted planar trees attached with one cycle, where edges on the tree are visited twice and the edge on the cycle is visited once. 
Therefore, the union graph belongs to $\mathcal{TC}^{k,h}$ in Definition~\ref{def: TC}.



Having discussed all possible choices of $v,e$ for the union graphs, we are ready to compute the limiting covariance when $k+h$ is even.

(i) $k, h \geq 3$ are odd. The only nonzero contribution in the limit comes from Case 2. 
Since the cycle will only be visited once in each $G_{I_k}$ and $G_{J_h}$, we have $\mathbb{E}[\overline{a}_{I_k}  ] = \mathbb{E} [\overline{a}_{J_h} ] = 0$. Then  the covariance given in \eqref{equ: dense_c_cov1} can be simplified as 
\begin{align}\label{equ: dense_c_cov_rootree}
     \mathrm{Cov}( X_{n,k}, X_{n,h})
      &= \frac{p}{(np)^{(k+h)/2}}  \sum_{G \in \mathcal{TC}^{k,h}} \sum_{l = (l_1, l_2, \dots, l_{(k+h)/2})}  \mathbb{E} \left[ \prod_{uv \in E({G})} \overline{a}_{l_u l_v}^2  \right]+o(1) \\
      &= \frac{p}{n^{(k+h)/2}}  \sum_{G \in \mathcal{TC}^{k,h}} \sum_{ l_1, \dots, l_{(k+h)/2}\in [n]}  \prod_{uv \in E(G)}   s_{l_u l_v} (1 - p s_{l_u l_v}) +o(1)\\
      &=p\sum_{G\in \mathcal{TC}^{k,h}} t(G, W_n') + o(1) \label{equ: covar_d_c_1_O}.
\end{align}
 Then,  if  $\delta_{\Box}(W_n', W') \rightarrow 0$ for some graphon $W$ as $n \rightarrow \infty$, by Lemma \ref{graphonconvergence}, we have
\begin{align}\label{equ: dense_c_var_pfcol}
      \mathrm{Cov}( X_{n,k}, X_{n,h}) 
      =  p \sum_{G\in \mathcal{TC}^{k,h}} t(G, W') + o(1).
 \end{align}

(ii) $k, h \geq 2$ are even. Then the nonzero contribution in the limit comes from Case 1 and Case 2. Case 2 gives the contribution in \eqref{equ: dense_c_var_pfcol}. Then, we only have Case 1 left to calculate. We have all possible union graphs $ G_{I_k \cup J_h} $ defined in Case 1 form $\mathcal T^{k,h}_1 $, and their contribution is given by
\begin{align}\label{equ: dense_c_cov_rootree1}
     & \frac{p}{(np)^{(k+h)/2}}  \sum_{T \in \mathcal T^{k,h}_2}  \sum_{l = (l_1, l_2, \dots, l_{(k+h)/2})} \mathbb{E} \bigg[ \overline{a}_{l_{u'} l_{v'}}^4 \prod_{uv \in E(T) \setminus \{{u' v'}\}} \overline{a}_{l_u l_v}^2    \bigg]  - \mathbb{E} \bigg[  \prod_{uv \in E({T)}  }   \overline{a}_{l_u l_v}^2 \bigg].  \\
\end{align}
where $u' v'$ is the additional edge of a multigraph in $\mathcal T^{k,h}_2$ that is visited 4 times. As $a_{ij} \sim \text{Ber}(p s_{ij})$ defined in Assumption~\ref{assumption1}, we have  \begin{equation}
        \mathbb{E}[  \overline{a}_{uv}^2] = p s_{uv} (1 - p s_{uv}), \quad \mathbb{E}[  \overline{a}_{uv}^4] = p s_{uv} (1 - p s_{uv}) \big(1 - 3p s_{uv} (1 - p s_{uv}) \big).
    \end{equation}
By plugging the moments and denote $   \sum_{l = (l_1, l_2, \dots, l_{(k+h)/2})} = \sum_{l}$ for simplicity, we can simplify \eqref{equ: dense_c_cov_rootree1} as 
\begin{small}
     \begin{align}
     &\frac{p}{(np)^{(k+h)/2}} \sum_{T \in \mathcal T^{k,h}_{1}}  \sum_{l} \left(  \left(1 - 3p s_{l_{u'} l_{v'}} (1 - p s_{l_{u'} l_{v'}}) \right) \prod_{uv \in E(T)} p s_{l_u l_v} (1 - p s_{l_u l_v})    \right) \\
        &\quad -  \left( p s_{l_{u'} l_{v'}} (1 - p s_{l_{u'} l_{v'}})  \prod_{uv \in E(T)} p s_{l_u l_v} (1 - p s_{l_u l_v})    \right)   \\
        &=\frac{p}{(np)^{(k+h)/2}} \sum_{T \in \mathcal T^{k,h}_{1}}  \sum_{l} \left( \prod_{uv \in E({T})} p s_{l_u l_v} (1 - p s_{l_u l_v}) \right)  \left(1 - 4p s_{l_{u'} l_{v'}} (1 - p s_{l_{u'} l_{v'}}) \right)  \\
        &= \frac{1}{n^{(k+h)/2}} \sum_{T \in \mathcal T^{k,h}_{1}}  \sum_{l} \left( \prod_{uv \in E({T})}  s_{l_u l_v} (1 - p s_{l_u l_v}) \right)  - 
        \frac{4p}{n^{(k+h)/2}} \sum_{T \in \mathcal T^{k,h}_{2}}  \sum_{l} \left( \prod_{uv \in E({T})}  s_{l_u l_v} (1 - p s_{l_u l_v}) \right)  \label{equ: dense_c_var}.
    \end{align}
    \end{small}
Since $s_e (1 - p s_e)$ is bounded for all edge $e$ from Assumption~\ref{assumption1}, with Definition~\ref{def: homomorphism density}, we can further write the covariance in terms of the  homomorphism density as follows \begin{align}
     \mathrm{Cov}( X_{n,k}, X_{n,h}) &= \frac{1}{n^{(k+h)/2}} \sum_{T \in \mathcal  T^{k,h}_{1}}  \sum_{  l_1, \dots , l_{(k+h)/2} \in [n]} \left( \prod_{e \in E({T}_{l})}  s_e (1 - p s_e) \right) \\
        &\quad  - 
        \frac{4p}{n^{(k+h)/2}} \sum_{T \in \mathcal  T^{k,h}_{2}}  \sum_{  l_1, \dots , l_{(k+h)/2} \in [n]} \left( \prod_{e \in E({T}_{l})}  s_e (1 - p s_e) \right) +o(1) \\
        &= \sum_{T \in \mathcal 
        T^{k,h}_{1}}  t(T,W'_n) -   4p \sum_{T \in \mathcal T^{k,h}_{2}}  t(T,W'_n)  +o(1). \label{equ: dense_c_tree_cov1}
 \end{align}
 
 By Lemma \ref{lem: W'}, if as $\delta_{1}(W_n, W) \rightarrow 0$ for some graphon $W$ as $n \rightarrow \infty$, we have $\delta_1(W_n',W')\to 0$. Then with Lemma~\ref{lem: graphon_multi}, $t(G,W_n')\to t(G,W')$ for any finite loopless multigraph $G$.
 Together with Case 2, we have when $k,h\geq 2$ are even, \begin{align}
     \mathrm{Cov}( X_{n,k}, X_{n,h}) =   \sum_{T \in \mathcal  T^{k,h}_{1}}  t(T,W') -   4p \sum_{T \in \mathcal T^{k,h}_{2}}  t(T,W') +  p \sum_{TC \in \mathcal{TC}^{k,h}} t( TC, W') +o(1).
    \end{align}

 \noindent \textbf{Step 2: Convergence of joint moments.}\label{step: 2_dense_c} 
 Next, we show $X_{n,k}$ are asymptotically jointly Gaussian. Since Gaussian distribution is characterized by its moments, it suffices to show convergence of joint moments of $L_{n,k}$. 
Let $k_1,\dots,k_t\geq 2$ be any $t$ positive integers independent of $n$. By Wick's formula \cite{wick1950evaluation}, it suffices to show as $n\to\infty$  
\begin{align}\label{equ: dense_nc_joint_momt}
   \mathbb E[ X_{n,k_1}\cdots X_{n,k_t}]\to \sum_{\pi\in \mathcal P_t} \prod_{(i,j)\in \pi} \Cov(X_{k_i},X_{k_j}),
\end{align}
where $\mathcal P_t$ is the set of all parings of $\{1,\dots,t\}$, i.e., all distinct ways of partitioning $[t]$ into pairs $\{i,j\}$. Note that when $t$ is odd, the right-hand side of \eqref{equ: dense_nc_joint_momt} is zero.

We first derive the alternative form of the joint moment that
\begin{align}
    \mathbb E[ X_{n,k_1}\cdots X_{n,k_t}] &=  \mathbb{E} \left[  (\sqrt{p})^{t} \cdot \prod_{i=1}^t \mathrm{tr}\left(\frac{\overline{A}_n}{\sqrt{np}}\right)^{k_i} - \mathbb{E}\mathrm{tr}\left(\frac{\overline{A}_n}{\sqrt{np}}\right)^{k_i}  \right] \\
    &= (\sqrt{p})^{t} \cdot \mathbb{E} 
    \sum_{I_{k_1}, I_{k_2}, \dots, I_{k_t}} \prod_{i=1}^t \left(  \frac{\overline{a}_{I_{k_i}}}{(np)^{k_i/2}} - \mathbb{E} \frac{\overline{a}_{I_{k_i}}}{(np)^{k_i/2}}  \right) \label{equ: dense_highm_c_b} .
\end{align}
where the sum is over $t$ tuples of $I_{k_1}$,\dots, $I_{k_t}$.  If $ k_i$ is odd for some $ i \in [n]$, then we have
\begin{align}\label{equ: dense_c_kodd}
    \mathbb E[ X_{n,k_1}\cdots X_{n,k_t}] = 0.
\end{align}

Then it remains to consider $ k_i$ is even for all $ i \in [n]$. Consider any tuple \( ( I_{k_1}, I_{k_2}, \dots, I_{k_t}) \) from \eqref{equ: dense_highm_c_b} and let \( G_{I_{k_1} \cup I_{k_2} \dots \cup I_{k_t}} \) denote its corresponding union graph with $v$ vertices and $e$ edges.  We introduce the following definition for a dependency graph. 
\begin{definition}[Dependency graph]\label{def: dependency_graph}
    Define \( H \) to be the dependency graph associated with \( G_{I_{k_1} \cup I_{k_2} \dots \cup I_{k_t}} \). Then each vertex $h_i$ of $H$ corresponds to a subgraph $ G_{I_{k_i}}$, and two vertices, $h_i$ and $h_j$ are connected by an edge if $G_{I_{k_i}}$ and $G_{I_{k_i}}$ are connected. Furthermore, we denote \( c \) as the number of connected components in \( H \). 
\end{definition}

   Denote $ v = \sum_{i = 1}^c v_{i}$ and $ e = \sum_{i = 1}^c e_{i}$. Specifically, each $e_i$ and $v_i$ is the number of vertices and edges of each independent subgraph. Thus, for each \( G_{I_{k_1} \cup I_{k_2} \dots \cup I_{k_t}} \) to have a nontrivial contribution, there can not exist any isolated point in \( H \) and each independent subgraph is nontrivial.

For each \( G_{I_{k_1} \cup I_{k_2} \dots \cup I_{k_t}} \) to have a nontrivial contribution, we require \begin{align}\label{equ: dense_c_gaussian_1}
   c \leq \left\lfloor t/2 \right\rfloor, \quad  v \leq  k/2  - t + 2c , \quad \text{and} \quad e \leq k/2  - t+c
\end{align}
where we denote $k = \sum_{i = 1}^t k_i$.
In addition, as $n \rightarrow \infty$, we have \begin{align}\label{equ: dense_c_gaussian_2}
     \mathbb  E[ X_{n,k_1}\cdots X_{n,k_t}]  = O( \frac{n^v}{n^{k/2}} \cdot \frac{p^e}{p^{k/2 - t/2}}),
\end{align}
then $ \mathbb E[ X_{n,k_1}\cdots X_{n,k_t}]$ remains non-negligible when $v \geq k/2 $.  

Together, it implies that $ \mathbb  E[ X_{n,k_1}\cdots X_{n,k_t}] $ is non-negligible when $ c = t/2$ and $t$ is even,
which means the dependency graph must form a perfect matching. Let $\mathcal P_t$ be the set of all parings of $\{1,\dots,t\}$, we have \begin{align}
    \mathbb  E[ X_{n,k_1}\cdots X_{n,k_t}] = \sum_{\pi\in \mathcal P_t} \prod_{(i,j)\in \pi} \Cov(X_{n,k_i},X_{n,k_j})+o(1).
\end{align}

Therefore, from Step 1, if $\delta_{\Box}(W_n', W') \rightarrow 0$ for some graphon $W'$ as $n \rightarrow \infty$, by Lemma \ref{lem: graphon_multi}, each $\Cov(X_{n,k_i},X_{n,k_j})$ converges to $\Cov(X_{k_i},X_{k_j})$. Then as $n\to\infty$  
\begin{align}\label{equ: dense_c_keven}
   \mathbb E[ X_{n,k_1}\cdots X_{n,k_t}] \to \sum_{\pi\in \mathcal P_t} \prod_{(i,j)\in \pi} \Cov(X_{k_i},X_{k_j}).
\end{align} 
Combining \eqref{equ: dense_c_kodd} and \eqref{equ: dense_c_keven}, the proof of Theorem~\ref{thm: centered_CLT_inhomo_dense} is complete.

\subsection{Proof of Theorem~\ref{thm: centered_CLT_inhomo_sparse1}}
The proof of Theorem~\ref{thm: centered_CLT_inhomo_sparse1} follows a similar approach to that of Theorem~\ref{thm: centered_CLT_inhomo_dense} with the new assumption of $p \to 0$ and $np \to \infty $. With the same linear statistics, we have 
\begin{align}
  \mathrm{Cov}( X_{n,k}, X_{n,h})= \frac{p}{(np)^{(k+h)/2}} \sum_{I_k,J_h} \mathbb{E}[\overline{a}_{I_k} \overline{a}_{J_h} ] - \mathbb{E}[\overline{a}_{I_k}  ] \mathbb{E} [\overline{a}_{J_h} ].
\end{align}

\textbf{Step 1: Computing the limiting covariance.}  For a given $G_{I_k\cup J_h}$, since the union graph must be connected to have a nontrivial contribution, it must satisfy \(v \leq e+1\). With Assumption \ref{assump: dense_1}, when $ v < (k+h)/2$, we have the contribution from all $G_{I_k\cup J_h}$ with a given $v,e$ is bounded by 
\begin{align}\label{equ: sparse_c_bound1}
        \frac{p}{(np)^{(k+h)/2}} \cdot n^v  ( Cp)^e  
        = O \left( \frac{ p^{e+1-v}}{(np)^{(k+h)/2 -v } } \right) = o(1).
    \end{align}
     Therefore, it remains to consider all union graphs with   $v-(k+h)/2 \geq 0$. Following  the same analysis  in the dense regime (see Step 1 in the proof of Theorem~\ref{thm: centered_CLT_inhomo_dense}),  the union graphs $G_{I_k \cup J_h}$ that contribute to the limiting variance must satisfy  \begin{align}
 e \leq v \leq e+1, \quad \text{and} \quad k+h \text{ is even}.
 \end{align}  
Therefore $v=e$ or $e+1$. We then consider 2 cases.

 \textbf{Case 1: \( v = e+1 \).} 
 In this case, $G_{I_k \cup J_h} $ is a tree formed by the union of two rooted planar trees $G_{I_k}$ and $G_{J_h}$ with overlapping edges. In this case, the right hand side of \eqref{equ: sparse_c_bound1} can be simplified to \begin{align}
     \frac{p}{(np)^{(k+h)/2}} \cdot n^v p^e  = \frac{1}{(np)^{(k+h)/2-v}} .
 \end{align}
Together with  \eqref{equ: dense_c_cov_ine_11}, we have $ \mathrm{Cov}( X_{n,k}, X_{n,h})$ remains nontrivial when \begin{align}
 (k+h)/2   \leq v = e+1 \leq  (k+h)/2.
\end{align}
Therefore, we have this case holds when $k,h \geq 2$ are even and $ e = v-1 = (k+h)/2 -1 $. Thus, the union graph belongs to $ \mathcal T^{k,h}_1$ defined in Definition~\ref{def: graph_tree}.

\textbf{Case 2: \(  v =e \).}
In this case, $G_{I_k \cup J_h}$ is a graph containing only one cycle.  From  \eqref{equ: dense_c_cov_ine_11}, we have $ e \geq (k+h)/2$. From  \eqref{equ: dense_c_cov_ine_22}, we have $ e \leq (k+h)/2$. Hence, $ v = e = (k+h)/2$. Therefore, by \eqref{equ: sparse_c_bound1}, we have \begin{align}
     \frac{p}{(np)^{(k+h)/2}} \cdot n^v p^e = \frac{p}{(np)^{(k+h)/2-v}} = p \to 0,
\end{align}
which implies the contribution in this case is zero in the limit for all $k,h$.

\smallskip

Together, having discussed all possible choices of $v,e$ for the union graphs, we have only Case 1 holds. Consequently, we can follow the same proof of  \eqref{equ: dense_c_tree_cov1}  to obtain that \begin{align}\label{equ: sparse_c_asy_cov}
      \mathrm{Cov}( X_{n,k}, X_{n,h}) =   \sum_{T \in \mathcal  T^{k,h}_{1}}  t(T,W_n) -   4p \sum_{T \in \mathcal 
 T^{k,h}_{2}}  t(T,W_n).
\end{align}

Since $ p \to 0$, the second term in \eqref{equ: sparse_c_asy_cov} vanishes as $ n \to \infty$. 
With only simple graphs in $\mathcal T_1^{k,h}$ left in the limit,  if $t(T,W_n)\to \beta_T$ for any finite tree $T$ as $n \rightarrow \infty$,  we have \begin{align}
     \mathrm{Cov}( X_{n,k}, X_{n,h}) =  \sum_{T \in \mathcal  T^{k,h}_{1}} \beta_T +o(1).
    \end{align}

\textbf{Step 2: Convergence of joint moments.}
First, we observe that \eqref{equ: dense_c_kodd} holds if there is at least one $k_i$ that is odd for some $i \in [n]$. In addition, from \eqref{equ: dense_highm_c_b} \eqref{equ: dense_c_gaussian_1}, and \eqref{equ: dense_c_gaussian_2}, we conclude that in the sparse regime, when $\mathbb E[ X_{n,k_1}\cdots X_{n,k_t}]$ is non-negligible, the dependency graph must still form a perfect matching. Consequently, Step 2 follows the same argument as in the dense regime. As $n \to \infty$,  
\begin{align}
    \mathbb E[ X_{n,k_1}\cdots X_{n,k_t}] \to \sum_{\pi\in \mathcal P_t} \prod_{(i,j)\in \pi} \Cov(X_{k_i},X_{k_j}).
\end{align}
This finishes the proof of Theorem~\ref{thm: centered_CLT_inhomo_sparse1}.

\subsection{Proof of Theorem~\ref{thm: inhomo_bounded_c}}

With the same linear statistics, we have 
\begin{align}\label{equ: bounded_cov_11}
  \mathrm{Cov}( X_{n,k}, X_{n,h})= \frac{p}{(np)^{(k+h)/2}} \sum_{I_k,J_h} \mathbb{E}[\overline{a}_{I_k} \overline{a}_{J_h} ] - \mathbb{E}[\overline{a}_{I_k}  ] \mathbb{E} [\overline{a}_{J_h} ].
\end{align}

For a given $G_{I_k\cup J_h}$, let $e$ be the number of distinct edges in $G_{I_k\cup J_h}$ and $v$ be the number of distinct vertices. Having $np$ being a constant and with Assumption~\ref{assump: dense_1}, when $ v < e+1$, we have the contribution from all $G_{I_k\cup J_h}$ with a given $v,e$ is bounded by \begin{align}
     \frac{p}{(np)^{(k+h)/2}} n^v(Cp)^e = O\left( \frac{n^{v-e-1} }{(np)^{(k+h)/2 -e -1} } \right)=O(n^{v-e-1}) =o(1).
\end{align}
Therefore, it remains to consider $ v \geq e+1$.
To ensure the connectivity of \( G_{I_k \cup J_h} \), we need \( v - 1 \leq e \) by the same argument from \eqref{equ: dense_c_cov_ine_22}.  Hence, for a nonzero contribution, each union graph must satisfy $v=e+1$,  which implies that \( G_{I_k \cup J_h} \) is a tree. 

Consider the individual graphs  $G_{I_k}, G_{J_h}$, then each $G_{I_k}, G_{J_h}$ is a tree and $I_k, J_h$ are closed walks on a tree of length $k,h$, respectively. If $k$ or $h$ is odd, such walks do not exist and the limiting contribution is $0$. Therefore, below we only consider the case when both $k,h$ are even.

Let $e_k,e_h$ denote the numbers of their edges, we have \[e_k=:i\leq \frac{k}{2}, \quad e_h=:j\leq \frac{h}{2},\] where $k$ and $h$ are both even. 
Additionally, to maintain the connectivity of the union graph $G_{I_k \cup J_h}$, it must hold that $e\leq i+j-1$.
Since $p=o(1)$, we have \(\mathbb{E}[{a}_{I_k}] \mathbb{E}[{a}_{J_h}]=o( \mathbb{E}[a_{I_k} a_{J_h} ])\), hence \eqref{equ: bounded_cov_11} can be simplified as 
\begin{align}
     \mathrm{Cov}(X_{n,k}, X_{n,h}) &= \frac{p}{(np)^{(k+h)/2}} \sum_{I_k,J_h} \mathbb{E}[a_{I_k} a_{J_h} ]+o(1).
\end{align}


Define \( \mathcal{T}_k \) to be the set of all rooted planar trees with \( k \) edges. 
we have the all $  G_{I_k \cup J_h}$ with nontrivial contribution can be expressed as gluing two trees $T_1\in \mathcal T_{i}$ and $T_2\in \mathcal T_{j}$ for some $i\leq k/2,j\leq l/2$. Recall the definition of $\mathcal{P}_{\#}(T_1, T_2)$  is defined in Definition~\ref{def:glue}. We can simplify the covariance and use the new notation defined above to obtain \begin{align}
    & \mathrm{Cov}(X_{n,k}, X_{n,h}) \\
     &= \frac{p}{(np)^{(k+h)/2}} \sum_{i\leq k/2, j\leq h/2}\sum_{T_1\in \mathcal T_i, T_2\in \mathcal T_j} ~\sum_{T \in \mathcal{P}_{\#}(T_1, T_2)} \sum_{l = (l_1, l_2, \dots, l_{v(T)})}  \prod_{e \in E(T)} ps_{l_e}(1-ps_{l_e}) + o(1)\\
   &= 
   \sum_{i\leq k/2, j\leq h/2}\sum_{T_1\in \mathcal T_i, T_2\in \mathcal T_j}  ~\sum_{T \in \mathcal{P}_{\#}(T_1, T_2)} \frac{1}{(np)^{(k+h)/2-v(T)}} \sum_{l = (l_1, l_2, \dots, l_{v(T)})}  \prod_{e \in E(T)} s_{l_e} +o(1) \\
   &=   \sum_{i\leq k/2, j\leq h/2}\sum_{T_1\in \mathcal T_i, T_2\in \mathcal T_j}  ~\sum_{T \in \mathcal{P}_{\#}(T_1, T_2)} \frac{1}{c^{(k+h)/2-v(T)}} t(T, W_n) + o(1).
\end{align}
 Assume for any finite tree $T$, $t(T, W_n)\to \beta_T$ as $n \rightarrow \infty$. Then, we have the covariance goes to $  \mathrm{Cov}(X_{k}, X_{h}) $ where 
 \begin{align}\label{equ: bounded_nc_cov_limit}
     \mathrm{Cov}(X_{k}, X_{h}) = \sum_{i\leq k/2, j\leq h/2}\sum_{T_1\in \mathcal T_i, T_2\in \mathcal T_j}  ~\sum_{T \in \mathcal{P}_{\#}(T_1, T_2)} \frac{\beta_T}{c^{(k+h)/2-v(T)}}+ o(1).
    \end{align}

The convergence of joint moments follows in the same way as in the proof of Theorem~\ref{thm: centered_CLT_inhomo_dense}, hence we omit the details. This finished the proof of Theorem~\ref{thm: inhomo_bounded_c}.

\section{Proof for the adjacency matrix}\label{sec:proof_adj}
\subsection{Proof of Theorem~\ref{thm:nc_CLT_inhomo_dense}}

The proof consists of two steps as follows.
We start with an asymptotic analysis of the covariance of $L_{n,k}, k\geq 2$.  In this proof, we will also introduce several definitions that will be repeatedly used in other theorems.

    We define a \( k \)-tuplet  as \( I_k = (i_1, i_2, \dots, i_k) \), where each \( i_j \) is an element of the set \( \{1, 2, \dots, n\} \). We define \( a_{I_k}=a_{i_1 i_2} a_{i_2 i_3} \dots a_{i_k i_1} \).
 Then, for any integer $ k,h \geq 2 $, we have the alternative form of the covariance that \begin{align}\label{equ: dense_nc_cov}
\mathrm{Cov}(L_{n,k}, L_{n,h}) = &\frac{n^2p^{1+\mathbf{1}\{k=2\}+\mathbf{1}\{h=2\}}}{(np)^{k+h}} \sum_{I_k,J_h} \mathrm{Cov}(a_{I_k}, a_{J_k}) \\
=&  \frac{n^2p^{1+\mathbf{1}\{k=2\}+\mathbf{1}\{h=2\}}}{(np)^{k+h}} \sum_{I_k,J_h} \mathbb{E}[a_{I_k} a_{J_h} ] - \mathbb{E}[a_{I_k}  ] \mathbb{E} [a_{J_h} ],
\end{align}
where the sum is over all $k$-tuple $i_1,\dots,i_k$ and $h$-tuple $j_1,\dots,j_h$.

    To each pair of $a_{I_k}, a_{J_h}$,  we associate a graph \( G_{I_k} \) to \( a_{I_k} \), constructed over the set of distinct indices  appearing in $I_k$ and edges formed by the unordered pairs $\{i_j,i_{j+1}\}$. Then $a_{I_k}$ is a closed walk on $G_{I_k}$ through all vertices. Similarly, We construct $ G_{J_h}$ associated to \( a_{J_h} \). We denote the union graph of $G_{I_k}$ and $ G_{J_h}$ by $G_{I_k\cup J_h}$. Then to find the asymptotic value of \eqref{equ: dense_nc_cov}, it suffices to estimate the contribution from all possible union graphs $G_{I_k\cup J_h}$.

\textbf{Step 1: Computing the limiting covariance.}\label{step: step1_dense_nc}
For a given $G_{I_k\cup J_h}$, let $e$ be the number of distinct edges in $G_{I_k\cup J_h}$ and $v$ be the number of distinct vertices. Then, having $p$ being a constant and with Assumption~\ref{assump: dense_1}, when $2+v<k+h$, we have the contribution from all $G_{I_k\cup J_h}$ with a given $v,e$ is bounded by 
\begin{align}\label{equ: dense_nc_cov_bound}
    \frac{n^2p }{(np)^{k+h}}n^v(Cp)^e=O(n^{2+v-k-h})=o(1).
\end{align}


Therefore, it remains to consider all union graphs with 
\begin{align}\label{eq:2v}
2+v\geq k+h.
\end{align}

    Since each entry in $A_n$ is independent, to have a nonzero $\mathrm{Cov}(a_{I_k}, a_{J_h})$,  $G_{I_k \cup J_h}$ must be  connected. Also, by Assumption \ref{assumption1} that $a_{ii}=0, i\in [n]$, $G_{I_k \cup J_h}$ cannot contain any loop. Let $e_{I_k}, e_{J_h}$ be the number of distinct edges in $G_{I_k},G_{j_h}$, respectively. 
These restrictions lead to 
\begin{align}\label{eq:3v}
v-1 \leq e \leq (k+h)-1.
\end{align}
 and \begin{align}\label{eq:eIJ_lowerbound}
(k+h)-1\geq  e_{I_k} + e_{J_h} -1 \geq e.
 \end{align}
 With \eqref{eq:2v} and \eqref{eq:3v}, the union graphs $G_{I_k \cup J_h}$ that contribute to the limiting variance must satisfy  \begin{align}\label{equ: dense_nc_ine}
 e-1 \leq v \leq e+1.
 \end{align}  
Therefore $v=e-1, e$, or $e+1$. We then consider 3 cases.

\smallskip 
 \textbf{Case 1: \( v  = e+1 \).}\label{case: dense_nc_case1}
 In this case,
 \( G_{I_k \cup J_h} \) is a tree formed by the union of two smaller trees, \( G_{I_k} \) and \( G_{J_h} \) with overlapping edges. At the same time, the fact that each \( a_{I_k}, a_{J_h} \) represents a closed walk on a tree requires every edge to be visited at least twice. This implies 
 \begin{align}\label{eq:eIJ}
     e_{I_k}\leq \frac{k}{2}, \quad e_{J_h}\leq \frac{h}{2}.
 \end{align}
From \eqref{eq:2v}, \eqref{eq:eIJ_lowerbound}, and \eqref{eq:eIJ}, we find 
\begin{align}
  k+h-3\leq e\leq   \frac{k+h}{2}-1,
\end{align}
 which gives $k+h\leq 4$. Thus, this case only holds when $k=h=2$, and the union graph is a $K_2$ in Definition~\ref{def:multi_cycle}.

\smallskip 
 \textbf{Case 2: $v = e $.}\label{case: dense_nc_case2}
 In this case, $G_{I_k \cup J_h}$ is a graph containing only one cycle. From \eqref{eq:2v} and \eqref{eq:3v}, we have $k+h-2\leq e\leq k+h-1$. Suppose $e=k+h-1$. Then from \eqref{eq:eIJ_lowerbound}, the only possibility is $e_{I_k}=k, e_{J_h}=h$ and $G_{I_k}, G_{J_h}$ has one overlapping edge. This implies $G_{I_k}$ is a cycle of length $k$ and $G_{J_h}$ is a cycle of length $h$, which creates 2 cycles in $G_{I_k \cup J_h}$, a contradiction. Therefore the only possibility is 
\begin{align}\label{eq:only_ekh}
e=k+h-2.
\end{align}
From \eqref{eq:eIJ_lowerbound}, we have in this case,
\begin{align}
  k+h-1 \leq   e_{I_k}+e_{J_h}\leq k+h.
\end{align}
If $e_{I_k}+e_{J_h}=k+h$, then both $G_{I_k}$ and $G_{J_h}$ are cycles and from \eqref{eq:only_ekh}, there are 2 overlapping edges in the union graph. This gives two cycles in $G_{I_k \cup J_h}$, a contradiction. Therefore we must have 
\begin{align}\label{eq:eIkeJh}
    e_{I_k}+e_{J_h}=k+h-1, 
\end{align}
and from \eqref{eq:only_ekh}, there is only one overlapping edge in $G_{I_k \cup J_h}$. 

Since the union graph contains only one cycle, without loss of generality,  we may assume $G_{I_k}$ contains a cycle and $G_{J_h}$ contains no cycle. Then $a_{J_h}$ is a closed walk on a tree, which gives $e_{J_h}\leq \frac{h}{2}$. 
Together with \eqref{eq:eIkeJh}, this implies $h=2,e_{J_h}=1$, and $e_{I_k}=e=k$. Hence both the union graph and $G_{I_k}$ are a cycle of length $k$ and $G_{J_h}$ is an edge. By symmetry, Case 2 only holds when $k\geq 3, h=2$ or  $k=2, h\geq 3$.

 \smallskip

\textbf{Case 3: $v  = e-1 $.} 
In this case,  the union graph contains 2 cycles. From \eqref{eq:2v}, we have 
    $e\geq k+h-1$. From \eqref{eq:eIJ_lowerbound}, we have $e\leq k+h-1$. Hence $e=k+h-1$ and $e_{I_k}=k, e_{J_h}=h$. Therefore, $G_{I_k}$, $G_{J_h}$ are two cycles and the union graph has one overlapping edge.
This case only holds when $k,h \geq 3$.


\smallskip

Having discussed all possible choices of $v,e$ for the union graphs, we are ready to compute the limiting covariance.

\smallskip

(i) $k=h=2$. The union graph is $K_2$ and there are two ways to form the labeled union graph from $G_{I_2}=K_2, G_{J_2}=K_2$.  By  \eqref{equ: dense_nc_cov}, we have \begin{align}
    \mathrm{Cov}(L_{n,2}, L_{n,2})  &= \frac{n^2p^3}{(np)^{4}} \cdot 2 \sum_{i,j} var(a_{ij}) = \frac{p^2}{(np)^{2}} \cdot 2\sum_{i,j} s_{ij}(1-ps_{ij})\\
    &=\frac{2}{n^2}\sum_{i,j}s_{ij}-\frac{2p}{n^2}\ \cdot  \sum_{i,j}s_{ij}^2 \\
    &=2 t(K_2, W_n)-2p \cdot t(C_2, W_n)\\
    &=2  t(K_2,W)-2p \cdot t(C_2, W)+o(1),
\end{align}
where the last line is due to Lemma~\ref{lem: graphon_multi} and the assumption $\delta_1(W_n,W)\to 0$.

(ii) $k=2, h\geq 3$. The only nonzero contribution in the limit comes from Case 2, where $G_{I_2}$ is an edge and $G_{J_h}$ is a cycle of length $h$. For a given labeling sequence of $G_{I_2}$ and $G_{J_h}$, there are $2h$ many ways to obtain $G_{I_2 \cup J_h}=C_h$ with labels, where $C_h$ is a cycle of length $h$.  Let $C_{2,h}$ be a cycle of length $h$ with one multi-edge attached.   Denote each unique labeling as $l = (l_1, l_2, \dots, l_{k+h-2})$, where $l_i\in [n]$ and each entry in $l$ is distinct.  Therefore we have
\begin{align}
    \mathrm{Cov}(L_{n,2}, L_{n,h})&=\frac{n^2p^2}{(np)^{2+h}}(2h)\sum_{l=(l_1,\dots,l_h)}  \prod_{uv\in E(C_h)}(ps_{l_ul_v})-ps_{l_1l_2}\prod_{uv\in E(C_h)}(ps_{l_ul_v}) +o(1)\\
    &=2h \left( t(C_h,W_n) - p \cdot t(C_{2,h}, W_n) \right) +o(1)\\
    &=2h\left( t(C_h,W) -p \cdot t(C_{2,h}, W) \right)+o(1),
\end{align}
where the last identity is due to Lemma~\ref{lem: graphon_multi}.

(iii) $k,h\geq 3$. The only nonzero contribution in the limit comes from Case 3. Given two cycles of length $k$ and $h$, there are $k$ and $h$ ways to select a shared edge in each graph and two possible directions to overlap. Thus, there are a total of $2kh$ ways to combine each pair of $G_{I_k} $  and $G_{J_h}$, which all result in the same unlabeled union graph $F_1^{k,h}$ (see Definition \ref{def: F-graph}). If we selected two shared vertices to overlap instead of edges, there are still $2kh$ combination ways and result in the same unlabeled union graph $F_2^{k,h}$ (see Definition \ref{def: F-graph}).

  Denote each unique labeling as $l = (l_1, l_2, \dots, l_{k+h-2})$,  where $l_i\in [n]$ and all $l_i$ are distinct. We can calculate the covariance using \eqref{equ: dense_nc_cov} that \begin{align}
    \mathrm{Cov}(L_{n,k}, L_{n,h})  
      &= \frac{n^2p}{(np)^{k+h}} \cdot 2kh \cdot \sum_{l = (l_1, l_2, \dots, l_{k+h-2})} \mathbb{E} \bigg[ \prod_{uv \in E(F^{k,h}_{1})} a_{l_ul_v}^2 \bigg] -  \mathbb{E} \bigg[ \prod_{uv \in E(F^{k,h}_{2})} a_{l_ul_v}^2 \bigg] +o(1) \\
    &= \frac{2kh}{n^{k+h-2}}  \sum_{l = (l_1, l_2, \dots, l_{k+h-2})} \bigg( \prod_{uv \in E(F^{k,h}_{1})} s_{l_ul_v} \bigg) -  p \bigg( \prod_{uv \in E(F^{k,h}_{2})} s_{l_ul_v} \bigg) +o(1) \label{equ: var_d_c_1}\\
     &= \frac{2kh}{n^{k+h-2}} \sum_{ l_1, \dots , l_{k+h-2} \in [n] } \bigg( \prod_{uv \in E(F^{k,h}_{1})} s_{l_ul_v} \bigg) -  p \bigg( \prod_{uv \in E(F^{k,h}_{2})} s_{l_ul_v} \bigg) + o(1), \label{equ: covar_d_nc_1_O}
\end{align}
where in \eqref{equ: covar_d_nc_1_O} we allow the coincidence of $ l_1, \dots , l_{k+h-2}$. By Definition~\ref{def: homomorphism density}, we have \begin{align}\label{equ: dense_nc_cov_densityhom}
     \mathrm{Cov}(L_{n,k}, L_{n,h}) = 2kh \left( t(F^{k,h}_{1}, W_n) - p \cdot t(F^{k,h}_{2}, W_n) \right) + o(1).
\end{align}
Therefore, if $\delta_1(W_n, W) \rightarrow 0$ for some graphon $W$ as $n \rightarrow \infty$, by Lemma \ref{lem: graphon_multi}, we have  \begin{align}
      \mathrm{Cov}(L_{n,k}, L_{n,h})=  2kh \left( t(F_1^{k,h}, W) - p \cdot t(F_2^{k,h}, W) \right)+o(1).
 \end{align}

 \textbf{Step 2: Convergence of joint moments.} \label{step: 2_dense_nc}
Next, we show $L_{n,k}$ are asymptotically jointly Gaussian. To analyze the joint moment, we first derive the alternative form of that
\begin{align}
    \mathbb E[ L_{n,k_1}\cdots L_{n,k_t}] &=  \mathbb{E} \left[ \left( n\sqrt{p} \right)^{t} \cdot p^{m} \cdot \prod_{i=1}^t \mathrm{tr}\left(\frac{A_n}{np}\right)^{k_i} - \mathbb{E}\mathrm{tr}\left(\frac{A_n}{np}\right)^{k_i}  \right]  \\
    &= \left( n\sqrt{p} \right)^{t} \cdot  p^{m} \cdot \mathbb{E} 
    \sum_{I_{k_1}, I_{k_2}, \dots, I_{k_t}} \prod_{i=1}^t \left(  \frac{a_{I_{k_i}}}{(np)^{k_i}} - \mathbb{E} \frac{a_{I_{k_i}}}{(np)^{k_i}}  \right)  \label{equ: dense_highm_nc_b},
\end{align}
where $ m = \#\{ k_i = 2\}$ and the sum is over $t$ tuples of $I_{k_1}$,\dots, $I_{k_t}$.

Consider any tuple \( ( I_{k_1}, I_{k_2}, \dots, I_{k_t}) \) from \eqref{equ: dense_highm_nc_b} and let \( G_{I_{k_1} \cup I_{k_2} \dots \cup I_{k_t}} \) denote its corresponding union graph with $v$ vertices and $e$ edges.  

 Define \( H \) to be the dependency graph associated with \( G_{I_{k_1} \cup I_{k_2} \dots \cup I_{k_t}} \) following the Definition~\ref{def: dependency_graph}.
 Denote $ v = \sum_{i = 1}^c v_{i}$ and $ e = \sum_{i = 1}^c e_{i}$. 
 Thus, for each \( G_{I_{k_1} \cup I_{k_2} \dots \cup I_{k_t}} \) to have a nontrivial contribution, there can not exist any isolated point in \( H \) and each independent subgraph is nontrivial.

For each \( G_{I_{k_1} \cup I_{k_2} \dots \cup I_{k_t}} \) to have a nontrivial contribution, we require \begin{align}\label{equ: dense_nc_high_uniongraph}
c \leq \left\lfloor t/2 \right\rfloor, \quad v \leq k -  2(t-c), \quad \text{and} \quad e \leq k - m - (t-c),
\end{align}
where we denote $k = \sum_{i = 1}^t k_i$.
In addition, as $n \rightarrow \infty$, we have \begin{align}\label{equ: dense_nc_high_O}
     \mathbb E[ L_{n,k_1}\cdots L_{n,k_t}] = O( \frac{n^v}{n^{k-t}} \cdot \frac{p^{e}}{p^{k-t/2-m}} ).
\end{align}
Then $\mathbb E[ L_{n,k_1}\cdots L_{n,k_t}]$ remains non-negligible when $v \geq k-t $. This occurs only when $ c = t/2$ and $t$ is even, resulting in $ v = k-t $. Therefore, it tells that when $ \mathbb E[ L_{n,k_1}\cdots L_{n,k_t}] $ is non-negligible, the dependency graph must form a perfect matching. Let $\mathcal P_t$ be the set of all parings of $\{1,\dots,t\}$, we have \begin{align}
    \mathbb E[ L_{n,k_1}\cdots L_{n,k_t}] = \sum_{\pi\in \mathcal P_t} \prod_{(i,j)\in \pi} \Cov(L_{n,k_i},L_{n,k_j}).
\end{align}

Therefore, from Step 1, if $\delta_1(W_n, W) \rightarrow 0$ for some graphon $W$ as $n \rightarrow \infty$, by Lemma \ref{lem: graphon_multi}, each $\Cov(L_{n,k_i},L_{n,k_j})$ converges to $\Cov(L_{k_i},L_{k_j})$. Then \eqref{equ: dense_nc_joint_momt} is achieved. This finishes the proof of Theorem~\ref{thm:nc_CLT_inhomo_dense}.



\subsection{Proof of Theorem~\ref{thm:nc_CLT_inhomo_sparse}}

\subsubsection{When $np\gg n^{1/2}$ and $p\to 0$}
Recall,
\begin{align}
\mathrm{Cov}(L_{n,k}, L_{n,h}) =  \frac{n^2p^{1+\mathbf{1}\{k=2\}+\mathbf{1}\{h=2\}}}{(np)^{k+h}} \sum_{I_k,J_h} \mathbb{E}[a_{I_k} a_{J_h} ] - \mathbb{E}[a_{I_k}  ] \mathbb{E} [a_{J_h} ].
\end{align}

Under Assumption~\ref{assump: dense_1} and Definition~\ref{def: union_graph}, consider all graphs \(G_{I_k \cup J_h}\) with given values of \(v\) and \(e\). Their contributions are bounded by  
\begin{align}\label{equ: sparse_nc_bound1}
    \frac{n^2 p^{1+\mathbf{1}\{k=2\}+\mathbf{1}\{h=2\}}}{(np)^{k+h}}n^v(Cp)^e 
= O\left(\frac{p^{e + \mathbf{1}\{k=2\}+\mathbf{1}\{h=2\} -1 - v}}{(np)^{k+h-2-v}}\right).
\end{align}
If \(v < e-1\), then for any \(G_{I_k \cup J_h}\) with nontrivial contribution, we have \(v < e-1 \leq (k+h)-2\). then since $np\to\infty$, the expression in \eqref{equ: sparse_nc_bound1} reduces to \(o(1)\). Additionally, since \(G_{I_k \cup J_h}\) is connected, it must satisfy \(v \leq e+1\). Therefore we have three cases: $v=e-1, e,$ or $e+1$.


 
 \smallskip  
 \textbf{Case 1: \( v= e+1 \).}
 In this case,
 \( G_{I_k \cup J_h} \) is a tree formed by the union of two smaller trees, \( G_{I_k} \) and \( G_{J_h} \) with overlapping edges. At the same time, each \( a_{I_k}, a_{J_h} \) represents a closed walk on a tree, and every edge is visited at least twice.  Together, this implies \begin{align}\label{equ: bounded_case_1_e}
     e \leq \frac{k+h}{2}-1,
 \end{align}
 and $k,h$ are even. Therefore, we analyze three cases.
 
\smallskip
 (i) $k=h=2$. Then \( e = v - 1 = 1 \), and \eqref{equ: sparse_nc_bound1} bounded by a constant. The corresponding union graph is a $K_2$ in Definition~\ref{def:multi_cycle}.

(ii) $k=2,h\geq 4$ and $h$ is even. Then \eqref{equ: sparse_nc_bound1} is bounded by 
\begin{align}
    o(p^{-1} n^{(e+1-h)/2})=o(p^{-1}n^{-1/2})=o(1).
\end{align}

 (iii) $k,h \geq 4$ and $k,h$ are even. Then as $ np\gg n^{1/2}$,  from   \eqref{equ: bounded_case_1_e}, we have  \eqref{equ: sparse_nc_bound1} is bounded by \begin{align}
      \frac{n^2 p}{(np)^{k+h}}n^v p^e \leq \frac{n^2}{(np)^{(k+h)/2}}=o(n^{2-\frac{k+h}{4}})=o(1).
 \end{align}

  \textbf{Case 2: $v = e$.} 
In this case, $G_{I_k \cup J_h}$ is a graph containing only one cycle. Thus, we have $k=2, h\geq 3$,  or $k,h \geq 3$. 

\smallskip
(i) $k=2, h\geq 3$. In this case, we have $ e = v = h$, then \eqref{equ: sparse_nc_bound1} is bounded by a constant. Hence, we have this case holds with the same corresponding union graph as in the dense case, which is a cycle of length $h$ denoted by $C_h$.

(ii) $k,h \geq 3$. In this case, $G_{I_k \cup J_h}$ is a graph containing only one cycle. Thus, either \( G_{I_k} \) and \( G_{J_h} \) share a cycle, meaning there are at least three overlapping edges, or they share a segment of a tree. The latter implies that the overlapping edge is traversed at least three times due to the fact that each \( a_{I_k}, a_{J_h} \) represents a closed walk.
In both cases, we have 
    $e \leq (k+h) -3$.
Then, as $ np \gg n^{1/2}$, from \eqref{equ: sparse_nc_bound1}, \begin{align}
     \frac{n^2 p^{1+\mathbf{1}\{k=2\}+\mathbf{1}\{h=2\}}}{(np)^{k+h}}n^v p^e 
     =\frac{n}{(np)^{k+h -e-1}}  = o(1).
\end{align}
Therefore, Case (ii) does not contribute to the limit.


  \textbf{Case 3: $v = e-1$.}
  In this case,  the union graph contains 2 cycles.
When $ e \leq  (k+h) -2$, from \eqref{equ: sparse_nc_bound1}, \begin{align}
     \frac{n^2 p^{1+\mathbf{1}\{k=2\}+\mathbf{1}\{h=2\}}}{(np)^{k+h}}n^v p^e 
     \leq  \frac{1}{(np)^{k+h -e-1}} = o(1).
\end{align}
Therefore, the case only holds when $ e = (k+h) -1 $, and $G_{I_k}$, $G_{J_h}$ are two cycles of length $k,h$, respectively with the union graph having exactly one overlapping edge. This only happens when $k,h\geq 3$.

\smallskip

Having discussed all possible choices of $v,e$ for the union graphs, we found the cases align precisely with the non-centered dense regime in Theorem~\ref{thm:nc_CLT_inhomo_dense}. Therefore, we can apply the results from Step 1 in the proof of Theorem~\ref{thm:nc_CLT_inhomo_dense}.
Under the new assumption of $ p \to 0$, we have the terms with an additional factor $p$ vanishes as $ n \to \infty$.  With only the existence of the simple graph left, we can relax the previous assumption to if $\delta_{\Box}(W_n,W)\to 0$ for some graphon $W$ as $n \rightarrow \infty$.  Then, by Lemma~\ref{graphonconvergence}, we have \begin{align}\label{equ: cov_np_gg_n1/2}
      \mathrm{Cov}(L_{n,k}, L_{n,h})= \begin{cases} 
    2 \cdot t(K_2,W) +o(1) & k=h=2\\
       2h \cdot  t(C_h, W) +o(1)   & k=2,h\geq 3\\
     2kh \cdot t(F_1^{k,h}, W) +o(1)   &    k,h\geq 3
      \end{cases}.
  \end{align}

From \eqref{equ: dense_highm_nc_b}, \eqref{equ: dense_nc_high_uniongraph}, and \eqref{equ: dense_nc_high_O}, we conclude that in the sparse regime, when \( \mathbb{E}[L_{n,k_1} \cdots L_{n,k_t}] \) is non-negligible, the dependency graph must still form a perfect matching. Consequently, the proof of convergence of joint moments follows the same argument as in the dense, non-centered regime (See Step 2  in the proof of Theorem~\ref{thm:nc_CLT_inhomo_dense}). As $n \to \infty$,  we have
\begin{align}
    \mathbb{E}[L_{n,k_1} \cdots L_{n,k_t}] \to \sum_{\pi \in \mathcal{P}_t} \prod_{(i,j) \in \pi} \mathrm{Cov}(L_{k_i}, L_{k_j}),
\end{align}
as desired.

\subsubsection{When $n^{1/2}p\to c $}
This case is similar to the regime when $np\gg n^{1/2}$ and we only address the differences. To compute the limiting covariance, we also have three cases $v=e-1,e,$ or $e+1$.

\textbf{Case 1}: $v=e+1$. The union graph is a tree.

(i) When $k=h=2$, the corresponding union graph is still $K_2$. 

(ii) When $k=2,h\geq 4$ and $h$ is even, then \eqref{equ: sparse_nc_bound1} is $O(1)$ only when $e=2, h=4$. We find the union graphs that give a nontrivial limiting covariance are all rooted planar trees with $h/2=2$  edges. Each rooted planar tree is counted $h$ many times due to the $h$ possible ways to combine $K_2$. Hence  when $k=2,h=4$, the contribution is 
\begin{align}
    &\frac{n^2p^2}{(np)^{2+4}} \cdot 4 \cdot p^{2}n^{3}\sum_{T\in \mathcal T_{2}}t(T,W_n)+o(1)=4 c^{-2} \sum_{T\in \mathcal T_2} t(T, W_n)+o(1).
\end{align}
And for $k=2, h\geq 6$, the contribution is zero.

(iii) When $k,h\geq 4, k,h$ are even.
The contribution is $O(1)$ only when $ k = h=4$ and $ v = e+1 = \frac{k+h}{2}$, and all possible union graphs $ G_{I_k \cup J_h} $ form $\mathcal T^{k,h}_1 = \mathcal T^{4,4}_1$. Hence, we have the contribution is \begin{align}
    &\frac{n^2p}{(np)^8} \cdot n^4 p^3 \sum_{T \in \mathcal  T^{4,4}_{1}}  t(T,W_n) +o(1) 
    = c^{-4}\sum_{T \in \mathcal  T^{4,4}_{1}}  t(T,W_n) +o(1).
\end{align}

\textbf{Case 2}: $v=e$. 

(i) When $k=2,h\geq 3$, the same conclusion holds as in the case for $np\gg n^{1/2}$.

(ii) $k,h\geq 3$. The only nonzero contribution comes from the case when $e=k+h-3$.
If  \( G_{I_k} \) and \( G_{J_h} \) share a segment of a tree, then $G_{I_k} \in \mathcal{TC}^{k}_r $ and $G_{J_h} $ is a rooted planar tree, where $ r \leq k-2 $ and $ h \geq 4$. Together,
\begin{align}
  e  \leq   e_{I_k} +   e_{J_h}-1 \leq \frac{k+r}{2} + \frac{h}{2} -1 
  \leq k + \frac{h}{2} -2.
\end{align}
Thus, this case does not hold for $ h \geq 4$.

If \( G_{I_k} \) and \( G_{J_h} \) share a cycle, this case only holds when $ k = h = 3$, where each \( G_{I_k} \) and \( G_{J_h} \) is a cycle of length $3$. Thus, there are total $3$ ways to choose the start edge and two directions to overlap, which all result in the same unlabeled union graph $C_3$ in Definition~\ref{def:multi_cycle}. Hence, the contribution is 
\begin{align}
    &\frac{n^2p}{(np)^{3+3}} \cdot 6 \cdot  p^3 n^3 \cdot t(C_3,W_n) 
= 6 c^{-2} \cdot t(C_3,W_n) + o(1) 
\end{align}

\textbf{Case 3}: $v=e-1$. The same argument for the case $np\gg n^{1/2}$ applies here and we obtain nontrivial covariance only when $k,h \geq 3$ and $e = k+h-1$.

Having discussed all possible choices of $v,e$ for the union graphs, we obtain the following limiting covariance:
\begin{align}
      \mathrm{Cov}(L_{n,k}, L_{n,h})= \begin{cases} 
    2 \cdot t(K_2,W) +o(1) & k=h=2\\[4pt]
       2h \cdot  t(C_h, W) +o(1)   & k=2,h = 3 \text{ or } h \geq 5\\[4pt]
     2h \cdot  t(C_h, W) + 4 c^{-2} \sum_{T\in \mathcal T_2} t(T, W)+o(1)  & k=2,h = 4 \\[4pt]
       6 c^{-2} \cdot t(C_3,W) + 2kh \cdot t(F_1^{k,h}, W) +o(1) &    k=h = 3 \\[4pt] 
    c^{-4}\sum_{T \in \mathcal  T^{k,h}_{1}}  t(T,W) + 2kh \cdot t(F_1^{k,h}, W) +o(1) &    k=h = 4 \\[4pt]
       2kh \cdot t(F_1^{k,h}, W) +o(1)   & \text{otherwise}   
      \end{cases}.
  \end{align}
Convergence of joint moments follows the same argument as in the proof of Theorem~\ref{thm:nc_CLT_inhomo_dense}.

\subsubsection{When $ n^{\frac{1}{m}} \ll np\ll n^{\frac{1}{m-1}}$ for $m\geq 3, m\in \mathbb Z$} \label{sec:nm}
We can expand the covariance as
\begin{align}
\mathrm{Cov}(L_{n,k}, L_{n,h}) =  \frac{n^2p^{1+\mathbf{1}\{k=2\}+\mathbf{1}\{ h=2\}}}{(np)^{k+h}} \sum_{I_k,J_h} \mathbb{E}[a_{I_k} a_{J_h} ] - \mathbb{E}[a_{I_k}  ] \mathbb{E} [a_{J_h} ].
\end{align}
Under Assumption~\ref{assump: dense_1} and Definition~\ref{def: union_graph}, consider all graphs \(G_{I_k \cup J_h}\) with given values of \(v\) and \(e\). Their contributions are bounded by  
\begin{align}\label{equ: sparse_nc_bound_m}
   O\left(\frac{p^{e + \mathbf{1}\{k=2\}+\mathbf{1}\{h=2\} -1 - v}}{(np)^{k+h-2-v}}\right).
\end{align}
Same as before, the only possible contribution in the limit comes from three cases: $v=e-1, e,$ or $e+1$.

 \smallskip  
 
 \textbf{Case 1: \( v= e+1 \).} In this case $k,h$ must be even, and $e\leq \frac{k+h}{2}-1$.

 (i) When $k=h=2$, the corresponding union graph is still $K_2$. 

(ii) When $k+h\geq 6$, from \eqref{equ: sparse_nc_bound_m},
the contribution is bounded by 
\begin{align}
   O\left(\frac{p^{\mathbf{1}\{k=2\}+\mathbf{1}\{h=2\}-2}}{(np)^{\frac{k+h}{2}-2}} \right)=o(1), 
\end{align}
where we use the condition $np\gg n^{1/m}$.

  \textbf{Case 2: $v = e$.} 
  In this case, $G_{I_k \cup J_h}$ is a graph containing only one cycle. Thus, we have $k=2, h\geq 3$,  or $k,h \geq 3$.

(i) $k=2, h\geq 3$. In this case, we have $ e = v = h$, 
then when $h\geq 2m-1$, \eqref{equ: sparse_nc_bound_m} is bounded by a constant, with the same corresponding union graph $C_h$.

(ii) $k,h \geq 3$.  
The contribution is 
\begin{align}\label{eq:p-1}
O\left(\frac{p^{-1}}{(np)^{k+h-2-e}} \right).
\end{align}
Let $\ell$ be the length of the cycle in $G_{I_k\cup J_h}$. In this case, $h$ must be even.
We consider two cases.
\begin{itemize}
    \item If $G_{I_k}$ contains a cycle and $G_{J_h}$ contains no cycles. Then  $e_{I_k} \leq \frac{k+\ell}{2}$, $e_{J_h} \leq \frac{h}{2}$, where $h$ must be even. And \[e\leq e_{I_k}+e_{J_k}-1 \leq \frac{k+h+\ell}{2}-1\leq k+\frac{h}{2}-1.\]
    Under the assumption $np\gg n^{1/m}$, when $h\geq 2m$,
    we have the contribution \eqref{eq:p-1} is  $O\left(\frac{n}{(np)^{h/2}} \right)=o(1)$. 
    \item If $G_{I_k}$ and $G_{J_h}$ share a cycle of length $\ell$, then $e_{I_k}\leq \frac{k+\ell}{2}, e_{J_h} \leq \frac{h+\ell}{2}$ and $e\leq e_{I_k}+e_{J_h}-\ell=\frac{k+h}{2}$. Then \eqref{eq:p-1} can be bounded by \[O\left(p^{-1} (np)^{2-\frac{k+h}{2}} \right)=O(p^{-1} (np)^{3-2m})=o(1)\]
    when $np\gg n^{1/m}$.
\end{itemize}

Therefore, Case (ii) does not contribute to the limit.


  \textbf{Case 3: $v = e-1$.}
 When $ e = (k+h) -1 $ and  $k,h\geq 2m-1$, this gives a nontrivial contribution in the limit, and $G_{I_k}$, $G_{J_h}$ are two cycles of length $k,h$, respectively with the union graph having exactly one overlapping edge. 

\smallskip

Having discussed all possible choices of $v,e$ for the union graphs,  if $\delta_{\Box}(W_n,W)\to 0$ for some graphon $W$ as $n \rightarrow \infty$.  we have \begin{align}\label{equ: cov_nc_sparse_1/m}
    \mathrm{Cov}(L_{n,k}, L_{n,h})= \begin{cases} 
     2 \cdot t(K_2,W) +o(1) & k=h=2\\
       2h \cdot  t(C_h, W) +o(1)   & k=2,h\geq 2m-1\\
     2kh \cdot t(F_1^{k,h}, W) +o(1)    &    k,h\geq 2m-1
      \end{cases}.
  \end{align}

Convergence of joint moments follows the same argument as in the proof of Theorem~\ref{thm:nc_CLT_inhomo_dense}, we omit the details.

\subsubsection{When $n^{1-\frac{1}{m}}p\to c$ with $m\geq 3,m\in \mathbb Z$}
This case is similar to the regime when $ n^{\frac{1}{m}} \ll np\ll n^{\frac{1}{m-1}}$ for $m\geq 3, m\in \mathbb Z$ and we only address the differences. We also have the three cases $v=e-1,e,$ or $e+1$ for all union graphs.

\textbf{Case 1: \( v= e+1 \).} In this case $k,h$ must be even.

 (i) When $k=h=2$, the corresponding union graph is still $K_2$. 
 
 (ii) When $k+h\geq 6$, then at least $k \geq 2m$ or $h \geq 2m$ since $k,h$ are even. We consider two cases. \begin{itemize}
  
     \item If $k=2$ and $ h \geq 2m$.
First, when $k+h > 2+2m$, by \eqref{equ: sparse_nc_bound_m}, the contribution is bounded by \begin{align}
O \left( \frac{1}{p \cdot (np)^{\frac{k+h}{2}-2}} \right) = O \left( \frac{n}{(c n)^{\frac{k+h}{2m}- \frac{1}{m}  }} 
\right)  = o(1),
     \end{align}
     where we use the condition that $ n^{1-\frac{1}{m}}p\to c$.
  When $k+h = 2+2m$, which only holds when $k = 2$ and $h = 2m$, \eqref{equ: sparse_nc_bound_m} is bounded by a constant, with the corresponding union graphs all rooted planar trees with edges $ h/2 = m$. Each rooted planar tree is counted $h$ many times due to the $h$ possible ways to combine $K_2$. Hence, the contribution is \begin{align}
     &\frac{n^2p^2}{(np)^{2+2m}} \cdot 2m \cdot p^{m} n^{m+1} \sum_{T\in \mathcal T_{m}}t(T,W_n)+o(1) 
     = \frac{2m}{c^m} \sum_{T\in \mathcal T_{m}}t(T,W_n)+o(1).
 \end{align}
\item If $ k,h \geq 2m$, then \eqref{equ: sparse_nc_bound_m} is bounded by \begin{align}
 O\left(\frac{n^2}{  (np)^{\frac{k+h}{2}} } \right) =  O\left(\frac{n^2}{  (cn)^{\frac{k+h}{2m}} } \right), 
\end{align}
where we use $ n^{1-\frac{1}{m}}p\to c$, and thus this case only holds when $ k = h = 2m$, and $ v-1 = e = \frac{k+h}{2} -1 = 2m-1$. In this case, $G_{I_k \cup J_h} $ is a tree formed by the union of two rooted planar trees $G_{I_k}$ and $G_{J_h}$ both with $m$ edges and an overlapping edge, as defined in Definition~\ref{def: graph_tree}. Hence, we have the contribution to be \begin{align}
    &\frac{n^2p}{(np)^{4m}} \cdot p^{2m-1} n^{2m} \sum_{T \in \mathcal  T^{2m,2m}_{1}}  t(T,W_n) +o(1) 
    = c^{-2m} \sum_{T \in \mathcal  T^{2m,2m}_{1}}  t(T,W_n) +o(1).
\end{align}
 \end{itemize}

 \textbf{Case 2 and 3: $v = e$ or $v=e-1$.}  We have the same conclusion holds as in Section~\ref{sec:nm}.

Therefore, the following limiting covariance holds for $m\geq 3$: 
\begin{footnotesize}
 \begin{align}
    \mathrm{Cov}(L_{n,k}, L_{n,h})= \begin{cases} 
    2 \cdot t(K_2,W) +o(1)  & k=h=2 \\[4pt]
    2h \cdot  t(C_h, W) +o(1)   & k=2,h =  2m-1 \text{ or } h \geq 2m+1 \\[4pt]
2m \cdot c^{-m} \sum_{T\in \mathcal T_{m}}t(T,W)+ 4m \cdot  t(C_{2m}, W) +o(1)  & k=2, h = 2m\\[4pt]
       c^{-2m} \sum_{T \in \mathcal  T^{2m,2m}_{1}}  t(T,W) +8m^2 \cdot t(F_1^{2m,2m}, W) +o(1)  & k = h = 2m \\[4pt]
     2kh \cdot t(F_1^{k,h}, W) +o(1)    &   \text{otherwise} 
      \end{cases}.
\end{align}
\end{footnotesize}
Convergence of joint moments follows the same argument as in the proof of Theorem~\ref{thm:nc_CLT_inhomo_dense}, we omit the details.

\subsubsection{When $np=n^{o(1)}$}
For any integer $k,h \geq 2$, we have the alternative form of the covariance  between $\widetilde{L}_{n,k}$ and  $\widetilde{L}_{n,h}$ given by \begin{align} \mathrm{Cov}(\widetilde{L}_{n,k}, \widetilde{L}_{n,h}) &= \frac{p}{(np)^{(k+h)/2}} \sum_{I_k,J_h} \mathbb{E}[a_{I_k} a_{J_h} ] - \mathbb{E}[ a_{I_k} ]\mathbb{E}[ a_{J_h}],
\end{align}where the sum is over all $k$-tuplet $I_k=(i_1,\dots,i_k)$ and $h$-tuplet $J_h=(j_1,\dots,j_h)$. 

\smallskip

\textbf{Step 1: Computing the limiting covariance.}

 For a given $G_{I_k\cup J_h}$, let $e$ be the number of distinct edges in $G_{I_k\cup J_h}$ and $v$ be the number of distinct vertices. Having $np = n^{o(1)}, np\to\infty$ and with Assumption~\ref{assump: dense_1}, when $ v < e+1$, we have the contribution from all $G_{I_k\cup J_h}$ with a given $v,e$ is bounded by 
\begin{align}\label{equ: sparse_npo(1)_O}
     \frac{p}{(np)^{(k+h)/2}} n^v(Cp)^e = O\left( \frac{n^{v-e-1} }{(np)^{(k+h)/2 -e -1} } \right) =o(1).
\end{align}

Then, it remains to consider $ v \geq e+1$. Together with the connectivity of \( G_{I_k \cup J_h} \) that $ v  \leq e +1 $, only the case $ v= e+1$ left, which further implies $ e \leq (k+h)/2-1$. Again, from \eqref{equ: sparse_npo(1)_O}, we have for a nonzero contribution, we require $ (k+h)/2 -e -1 \leq 0$. 

Therefore, each union graph must satisfy $v=e+1 = (k+h)/2 $, which aligns with Case 1 in the dense, centered regime (See Case 1 in the proof of Theorem~\ref{thm: centered_CLT_inhomo_dense}). Apply the same notation defined in Definition~\ref{def: graph_tree}, we have \begin{align}
    \mathrm{Cov}(\widetilde{L}_{n,k}, \widetilde{L}_{n,h}) = \sum_{T \in \mathcal 
        T^{k,h}_{1}}  t(T,W_n) -   p \sum_{T \in \mathcal T^{k,h}_{2}}  t(T,W_n)  +o(1).
\end{align}
Since $ p \to 0$, the second term vanishes as $n\to\infty$. if $t(T,W_n)\to \beta_T$ for any finite tree $T$,  then,  we have \begin{align}\label{equ: sparse_npo(1)_cov_limit}
    \mathrm{Cov}(\widetilde{L}_{k}, \widetilde{L}_{h}) = \sum_{T \in \mathcal 
        T^{k,h}_{1}} \beta_T +o(1).
\end{align}

\textbf{Step 2: Convergence of joint moments.}\label{step: step2_sparse_o(1)}
Next, we show $\widetilde{L}_{n,k}$ are asymptotically jointly Gaussian. To analyze the joint moment, we first derive the alternative form of that
\begin{align}
    \mathbb E[ \widetilde{L}_{n,k_1}\cdots \widetilde{L}_{n,k_t}] &=  \mathbb{E} \left[ \left( \sqrt{p} \right)^{t}  \cdot \prod_{i=1}^t \mathrm{tr}\left(\frac{A_n}{\sqrt{np}}\right)^{k_i} - \mathbb{E}\mathrm{tr}\left(\frac{A_n}{\sqrt{np}}\right)^{k_i}  \right] \label{equ: dense_highm_o} \\
    &= \left( \sqrt{p} \right)^{t} \cdot \mathbb{E} 
    \sum_{I_{k_1}, I_{k_2}, \dots, I_{k_t}} \prod_{i=1}^t \left(  \frac{a_{I_{k_i}}}{(np)^{k_i/2}} - \mathbb{E} \frac{a_{I_{k_i}}}{(np)^{k_i/2}}  \right)  \label{equ: bounded_highm_b},
\end{align}
where the sum is over $t$ tuples of $I_{k_1}$,\dots, $I_{k_t}$.  Consider any tuple \( ( I_{k_1}, I_{k_2}, \dots, I_{k_t}) \) from \eqref{equ: bounded_highm_b} and let \( G_{I_{k_1} \cup I_{k_2} \dots \cup I_{k_t}} \) denote its corresponding union graph with $v$ vertices and $e$ edges.

Then by Definition~\ref{def: dependency_graph} and the same analysis as in \eqref{equ: dense_c_gaussian_1}, we have for each \( G_{I_{k_1} \cup I_{k_2} \dots \cup I_{k_t}} \) to have a nontrivial contribution, we require $ c \leq \left\lfloor t/2 \right\rfloor$ and $ v \leq e + c$. In addition, as $n \rightarrow \infty$,  \begin{align}
      E[\widetilde{L}_{n,k_1}\cdots \widetilde{L}_{n,k_t}] 
      = O \left( \frac{n^{v-t/2 - e }}{ {(np)}^{k/2 - t/2 - e}} \right),
\end{align}
which remains non-negligible when $v \geq  t/2+e $. 
Together, it implies that when $  \mathbb E[ \widetilde{L}_{n,k_1}\cdots \widetilde{L}_{n,k_t}]$ is non-negligible, the dependency graph must form a perfect matching, that is $t$ is even and \begin{align}
    c = t/2, \quad v =t/2+e, \quad \text{and} \quad e =k/2 - t/2.
\end{align} 
Therefore, using the result in \eqref{equ: sparse_npo(1)_cov_limit} from Step 1, we have \begin{align}
    E[ \widetilde{L}_{n,k_1}\cdots \widetilde{L}_{n,k_t}]  = \sum_{\pi\in \mathcal P_t} \prod_{(i,j)\in \pi} \Cov(\widetilde{L}_{n,k_i},\widetilde{L}_{n,k_j}) \to \sum_{\pi\in \mathcal P_t} \prod_{(i,j)\in \pi} \Cov(\widetilde{L}_{k_i},\widetilde{L}_{k_j}).
\end{align}
This finishes the proof.

\subsection{Proof of Theorem~\ref{thm:nc_CLT_inhomo_bounded}}
This proof is similar to the proof of Theorem~\ref{thm: inhomo_bounded_c}, and we only address the differences.
For any integer $k,h \geq 2$, we have the alternative form of the covariance  between $\widetilde{L}_{n,k}$ and  $\widetilde{L}_{n,h}$ given by \begin{align}\label{equ: bounded_cov_1} \mathrm{Cov}(\widetilde{L}_{n,k}, \widetilde{L}_{n,h}) &= \frac{p}{(np)^{(k+h)/2}} \sum_{I_k,J_h} \mathbb{E}[a_{I_k} a_{J_h} ] - \mathbb{E}[ a_{I_k} ]\mathbb{E}[ a_{J_h}].
\end{align}
Then, all $  G_{I_k \cup J_h}$ with nontrivial contribution can be expressed as gluing two trees $T_1\in \mathcal T_{i}$ and $T_2\in \mathcal T_{j}$ for some $i\leq k/2,j\leq l/2$. Recall the definition of $\mathcal{P}_{\#}(T_1, T_2)$   in Definition~\ref{def:glue}. We can simplify the covariance and use the new notation defined above to obtain \begin{align}
     \mathrm{Cov}(\widetilde{L}_{n,k}, \widetilde{L}_{n,h}) &= \frac{p}{(np)^{(k+h)/2}} \sum_{i\leq k/2, j\leq h/2}\sum_{T_1\in \mathcal T_i, T_2\in \mathcal T_j} ~\sum_{T \in \mathcal{P}_{\#}(T_1, T_2)} \sum_{l = (l_1, l_2, \dots, l_{v(T)})}  \prod_{e \in E(T)} ps_{l_e} + o(1)\\
   &= 
   \sum_{i\leq k/2, j\leq h/2}\sum_{T_1\in \mathcal T_i, T_2\in \mathcal T_j}  ~\sum_{T \in \mathcal{P}_{\#}(T_1, T_2)} \frac{1}{(np)^{(k+h)/2-v(T)}} \sum_{l = (l_1, l_2, \dots, l_{v(T)})}  \prod_{e \in E(T)} s_{l_e} +o(1) \\
   &=   \sum_{i\leq k/2, j\leq h/2}\sum_{T_1\in \mathcal T_i, T_2\in \mathcal T_j}  ~\sum_{T \in \mathcal{P}_{\#}(T_1, T_2)} \frac{1}{c^{(k+h)/2-v(T)}} t(T, W_n) + o(1).
   \label{equ: bounded_c_var}
\end{align}
 Assume for any finite tree $T$, $t(T, W_n)\to \beta_T$ as $n \rightarrow \infty$. Then we have the covariance goes to $ \mathrm{Cov}(\widetilde{L}_k, \widetilde{L}_h) $ where \begin{align}
      \mathrm{Cov}(\widetilde{L}_k, \widetilde{L}_h) = \sum_{i\leq k/2, j\leq h/2}\sum_{T_1\in \mathcal T_i, T_2\in \mathcal T_j}  ~\sum_{T \in \mathcal{P}_{\#}(T_1, T_2)} \frac{\beta_T}{c^{(k+h)/2-v(T)}}+ o(1).
    \end{align}
With the same linear spectral statistics $ \widetilde{L}_{n,k}$ and under the assumption of $np = c$, we can repeat the proof from the sparse, non-centered, $np = n^{o(1)} 
$ regime (See the proof of Theorem~\ref{thm:nc_CLT_inhomo_sparse}) to establish the convergence of joint moments. This finishes the proof of Theorem~\ref{thm:nc_CLT_inhomo_bounded}.

\subsection*{Acknowledgements}
    Y. Zhu was partially supported by an AMS-Simons Travel Grant. This material is based upon work supported by the Swedish Research Council under grant no. 2021-06594 while Y. Zhu was in residence at Institut Mittag-Leffler in Djursholm, Sweden during the Fall of 2024.

\bibliographystyle{plain}
\bibliography{ref}

\appendices

\section{Proof of auxiliary  lemmas}\label{sec:appendix}

\subsection{Proof of Lemma~\ref{lem: graphon_multi}}

The proof is similar to  \cite[Theorem 3.7]{borgs2008convergent} and we prove a stronger statement. 
\begin{lemma}
    Let $F=(V,E)$ is be a finite loopless multigraph, then  for any $W, W'\in \mathcal W_0$,
    \begin{align}
        |t(F,W)-t(F,W')|\leq |E| \delta_1(W,W').
    \end{align}
\end{lemma}

\begin{proof}
Let $V=[n]$ and $E=\{e_1,\dots, e_m\}$. Let $e_t=\{i_t,j_t\}$ for $t\in [m]$. Then 
\begin{align}
    &t(F,W)-t(F,W')=\int_{[0,1]^n} \prod_{\{i_t,j_t\}\in E}W(x_{i_t},x_{j_t})-\prod_{\{i_t,j_t\}\in E}W'(x_{i_t},x_{j_t}) dx_1\cdots dx_n\\
    &=\sum_{t=1}^m \int_{[0,1]^n}\prod_{s<t}W(x_{i_s}, x_{j_s}) \prod_{s>t} W'(x_{i_s}, x_{j_s}) (W(x_{i_t},x_{j_t})-W'(x_{i_t},x_{j_t})) dx_1\cdots dx_n.
\end{align}
Therefore, by triangle inequality, 
$|t(F,W)-t(F,W')| \leq m \|W-W'\|_1$.
Since $t(F,W)=t(F,W^{\phi})$ for any measure-preserving   bijection $\phi$ between $[0,1]$ and $[0,1]$, we have 
\begin{align}
  |t(F,W)-t(F,W')| \leq m\inf_{\phi}\|W^{\phi}-W'\|_1=|E| \delta_1(W,W').  
\end{align}
This finishes the proof.
\end{proof}

\subsection{Proof of Lemma~\ref{lem: W'}}
\begin{proof}
    Let $\sigma$ be any measure-preserving bijection $[0,1]\to [0,1]$. Then
    \begin{align}
    &  \int_{[0,1]^2} |W_n'^{\sigma}(x,y)-W'(x,y)|dxdy\\
      = & \int_{[0,1]^2}|W^{\sigma}(x,y)(1-pW^{\sigma}(x,y))-W(x,y)(1-pW(x,y))| dxdy\\
        \leq & \int_{[0,1]^2} |(W_n^{\sigma}(x,y)-W(x,y))(1-pW_n^{\sigma})|dxdy +p\int_{[0,1]^2} |W(x,y)(W(x,y)-W_n^{\sigma}(x,y)|dxdy\\
        \leq & 2 \int_{[0,1]^2} |(W_n^{\sigma}(x,y)-W(x,y))|dxdy.
    \end{align}
    Taking infimum over all $\sigma$, we have 
    $\delta_1(W_n', W') \leq 2\delta_1(W_n, W)$. Therefore the conclusion holds.
\end{proof}

\section{The size of $\mathcal T_{1}^{k,h}, \mathcal T_{2}^{k,h}$, and $\mathcal{TC}^{k,h}$}\label{sec:size}

In this section, let $C_k=\frac{1}{k+1}\binom{2k}{k}$ be the $k$-th Catalan number, which counts the number of all rooted planar trees with $k$ edges.
We first compute the size of $\mathcal T_{1}^{k,h}, \mathcal T_{2}^{k,h}$ introduced in Definition~\ref{def: graph_tree}.
\begin{prop}
For any even integers $k,h\geq 2$, we have  \begin{align}
    \operatorname{card}(\mathcal 
 T^{k,h}_{1}) = \operatorname{card}(\mathcal  T^{k,h}_{2} ) = \frac{kh}{2} C_{k/2}  C_{h/2} .
\end{align}
\end{prop}
\begin{proof}
 To construct a graph in $\mathcal 
 T^{k,h}_{1}$, we first choose two rooted planar trees with $k/2,h/2$ edges, respectively. There are $C_{k/2}\cdot C_{h/2}$ many choices. Next, there are total $ k/2 \cdot h/2$ ways to select one overlapping edge, with two directions for them to overlap.  The same argument works for $\mathcal 
 T^{k,h}_{2}$.
\end{proof}

To compute the size of $\mathcal{TC}^{k,h}$ in Definition~\ref{def: TC}, we need the following definition.
    \begin{definition}\label{def: s(k,h)}
 For positive \( r \) and non-negative \( s \), denote by \( \mathcal{P}_r(s) \) the set of all partitions \( \kappa \) into \( s \) with \( r \) parts: \( \kappa = (\kappa_1, \ldots, \kappa_r) \) with \( \sum_{i=1}^r \kappa_{i} = s \) and each \( \kappa_{i} \geq 0 \). And define
    \[
    P_r(s) = \sum_{\kappa \in \mathcal{P}_r(s)} \prod_{i=1}^r C_{\kappa_{i}}.
    \]
 For any positive integer \( m \), define the set \( \mathcal{I}(m) = \{ r \equiv m \pmod{2} \mid 3 \leq r \leq m \} \). When  \( k + h \) is even, set
    \[
    S(k, h) = \sum_{r \in \mathcal{I} \left( \frac{k + h}{2} \right)} \frac{2 k h}{r} P_r \left( \frac{k - r}{2} \right) P_r \left( \frac{h - r}{2} \right).
    \]
\end{definition}

\begin{prop}
Assume $k+h\geq 2$ is even. Then 
    $\operatorname{card}({\mathcal{TC}}^{k,h}) =S(k,h)$. 
\end{prop}

\begin{proof}
Consider graphs in $\mathcal{TC}^{k,h}$ with a cycle of length $r$. Then $r$ must be chosen from the set
\begin{align}
    \mathcal{I} \left( \frac{k + h}{2} \right) = \left\{ r \equiv \frac{k + h}{2} \pmod{2} \mid 3 \leq r \leq \frac{k + h}{2} \right\},
\end{align}
since each edge not in the cycle must be traversed at least twice. Let $G_{I_k}, G_{J_h}$ be two graphs spanned by a closed walk of length $k$ and $h$, respectively, with a shared cycle of length $r$.

For $G_{I_k}$,
there are \( (k - r) / 2 \)  edges not in the cycle of length $r$. Define the number of edges as \( \kappa = (\kappa_1, \ldots, \kappa_r) \) with \( \sum_{i=1}^r \kappa_{i} = (k - r) / 2 \) and \( \kappa_{i} \geq 0 \) for each \( i \). Here, \( \kappa_{i} \) represents the number of edges in each tree attached to the \( i \)-th vertex of the cycle, and there are a total of \( r \) trees, one for each vertex of the cycle.   Thus, the total number of ways to form these trees is given by
\begin{align}
     \sum_{\kappa \in \mathcal{P}_r\left( \frac{k - r}{2} \right)} \prod_{i=1}^r C_{\kappa_{i}} = {P}_r\left( \frac{k - r}{2} \right),
\end{align}
where we use Definition~\ref{def: s(k,h)}.
Then, the total number of ways to form the trees in $G_{I_k}$ is given by $ {P}_r\left( \frac{h - r}{2} \right) $.

To construct the union graph \( G_{I_k \cup J_h} \) from \( G_{I_k} \) and \( G_{J_h} \), we first choose the starting point in both \( G_{I_k} \) and \( G_{J_h} \) for the cycle, which provides \( k \cdot h \) possible pairs of starting vertices. Additionally, we have two choices for the direction of the cycle in each graph (clockwise or counterclockwise), resulting in \( 2 \) choices per graph.
To account for rotational symmetry, each pair of walks is equivalent under any simultaneous cyclic permutation, leading to an overcount by a factor of \( r \). Therefore, the total number of ways to construct \( G_{I_k \cup J_h} \) with a given cycle length \( r \) is

\[
\frac{2 k h}{r} P_r\left( \frac{k - r}{2} \right) P_r\left( \frac{h - r}{2} \right).
\]
Summing over all possible values of \( r \) in \( \mathcal{I} \left( \frac{k + h}{2} \right) \), we obtain
\begin{align}
    \operatorname{card}({\mathcal{TC}}^{k,h}) = \sum_{r \in \mathcal{I} \left( \frac{k + h}{2} \right)} \frac{2 k h}{r} P_r \left( \frac{k - r}{2} \right) P_r \left( \frac{h - r}{2} \right)
    = S(k,h).
\end{align}
\end{proof}
\end{document}